\documentclass[11pt,letterpaper]{article}
\pdfoutput=1
\usepackage[english]{babel}
\usepackage{amsmath}
\usepackage{amsfonts, mathrsfs}
\usepackage{amssymb}
\usepackage{verbatim}
\usepackage{graphicx}
\usepackage{color}
\usepackage{url}
\usepackage{tikz}
\usepackage{float}
\usepackage{graphicx}
\usepackage{multirow}
\usepackage{epsfig}
\usepackage{epstopdf}
\usepackage{verbatim}
\usepackage{enumerate}
\usepackage{subcaption}
\usepackage[linesnumbered,ruled]{algorithm2e}

\allowdisplaybreaks

\numberwithin{equation}{section}
\newtheorem{theorem}{Theorem}[section]

\newtheorem{proposition}[theorem]{Proposition}
\newtheorem{lemma}[theorem]{Lemma}

\newtheorem{definition}[theorem]{Definition}
\newtheorem{remark}[theorem]{Remark}

\newenvironment{proof}[1][Proof]{\textbf{#1.} }{\ \rule{0.5em}{0.5em}}
\definecolor{mypink1}{rgb}{0.858, 0.188, 0.478}

\title{Last-mile shared delivery:\\ A discrete sequential packing approach}
\author{
	Junyu Cao, Mariana Olvera-Cravioto, Zuo-Jun (Max) Shen \\ {\small Industrial Engineering and Operations Research, University of California, Berkeley}}
\begin{document}
\maketitle

\begin{abstract}
We propose a model for optimizing the last-mile delivery of $n$ packages, from a distribution center to their final recipients, using a strategy that combines the use of ride-sharing platforms (e.g., Uber or Lyft) with traditional in-house van delivery systems. The main objective is to compute the optimal reward offered to private drivers for each of the $n$ packages, such that the total expected cost of delivering all packages is minimized. Our technical approach is based on the formulation of a {\em discrete sequential packing} problem, where bundles of packages are picked up from the warehouse at random times during the  interval $[0, T]$. Our theoretical results include both exact and asymptotic (as $n \to \infty$) expressions for the expected number of packages that will be picked up by time $T$, and are closely related to the classical R\'enyi's parking/packing problem.

\end{abstract}

\section{Introduction}

Last-mile delivery is arguably one of the biggest challenges in logistics management, and it may take up to 28\% of the total transportation costs~\cite{roca2016evaluation}.  Therefore, minimizing the cost of this last step of the order fulfillment process is of great importance. Historically, the delivery of items/packages to their final recipients from a centrally located warehouse has been done by either investing in a fleet of vans or trucks to operate regionally, or by outsourcing to third-party logistics companies specializing in such services. However, today's popular ``sharing economy" has provided a new alternative that can take advantage of ride-sharing platforms to offset some of the last-mile delivery costs. In this paper, we propose a model for analyzing a delivery process that combines the use of private drivers (e.g., Uber or Lyft drivers) with a more traditional van delivery system, with the idea of reducing the need to invest and maintain a large fleet of delivery vans or trucks, which is known to be very costly. 

We study here a problem where a warehouse or distribution center has to deliver a large number of packages on a given day. Traditionally, this would be done by deploying a fleet of vans along many different routes, with the routes being computed efficiently using some version of the capacitated vehicle routing problem (CVRP). One of the main features to focus on regarding this approach is the number of vans that are needed to deliver all the packages during the allotted time. Instead, our proposed strategy consists on setting up a window of time during the first part of the day, say $[0, T]$, during which private drivers are encouraged to pick up bundles of packages from the warehouse and deliver them themselves, with any remaining packages at time $T$ being delivered by the warehouse's van system.  Depending on the delivery deadlines that the warehouse has to meet and the feasible delivery times for vans, we could think of time $T$ as the early afternoon, which makes the vans operate during the late afternoon and until the early evening (e.g., UPS and Amazon vans deliver packages until around 8 pm), or even as the end of the working day, in which case the remaining packages are delivered the next day and vans work a full day schedule.  The key difference between our proposed approach compared to the traditional van-only one, is that the number of vans needed will be considerably smaller, which translates into important savings for the warehouse. 

The main challenges of using a mixed strategy are: 1) the modeling of the behavior of private drivers, and 2) the design of a payment scheme that will incentivize private drivers to pick up packages from the warehouse efficiently. Our proposed approach keeps the computational cost low by computing the payment reward that the warehouse will offer for each individual package only once, at the beginning of the day, based on the expected number of packages that will be picked up by private drivers during the time window $[0, T]$.  The main mathematical contribution of this paper lies precisely on the computation of the latter. The key idea is that the number of packages that can be picked up by private drivers can be modeled in a way that is independent of their destinations, provided that all packages are equally desirable to the drivers. To achieve this, we design a payment scheme that makes the profit for delivering packages the same across packages, and we control this profit through a common ``incentive rate". Once the expected number of packages that can be picked up during $[0,T]$ is computed (as a function of the incentive rate) we find the optimal incentive rate  by solving a single-variable continuous optimization problem. The daily computations are done in two phases: one at the beginning of the day that ends with the assignment of the payment rewards for each individual package using the optimal incentive rate, and another one at time $T$ that computes an optimal solution to the CVRP for the leftover packages. 

As mentioned above, the technical contribution of the paper is the computation of the expected number of packages that can be picked-up by private drivers during the time period $[0, T]$. Our methodology is closely related to the analysis of R\'enyi's classical ``parking/packing problem" \cite{renyi1958one}, which studies a model where unit-length ``cars" arrive sequentially to random points of a ``sidewalk" of fixed length and park if they can ``fit", with no parked cars ever leaving. The main question studied in \cite{renyi1958one} is the distribution of the unoccupied space once no more cars can fit, i.e., when there are no remaining unoccupied subintervals of length one. In the context of our package delivery process, we consider a discrete version of the problem where the cars are replaced by bundles of packages and parking attempts are replaced by bundle requests from the private drivers. Since we need to model the randomness in the times at which private drivers' requests arrive, we incorporate a time element to our formulation in the spirit of the packing problem considered in \cite{coffman1998packing}, where the cars/packages (requests for bundles of packages in our case) arrive according to a Poisson process. Our theoretical results are therefore of independent interest within the literature on R\'enyi's parking/packing problem, and include a novel convergence rate theorem that was previously unknown even for special cases of our model.

The paper is organized as follows. In Section~\ref{S.LiteratureReview} we include an overview of the existing literature on last-mile delivery problems in a shared economy.  In Section~\ref{MainResult} we introduce a {\em discrete sequential packing} (DSP) problem  that will be used to compute the expected number of packages that can be picked-up from the warehouse by private drivers during the time period $[0, T]$. In Section~\ref{S.Pricing} we describe a way to estimate the cost of delivering $n$ packages with known destinations $\{ {\bf x}_1, \dots, {\bf x}_n\} \subseteq \mathbb{R}^2$, using a strategy that combines the use of private drivers and in-house vans. We also include in that section a comparison between our proposed strategy and a more traditional van-only strategy. Section~\ref{S.Proofs} contains all the proofs of our theoretical results. Finally, to illustrate our methodology we include, a numerical experiments section (Section~\ref{S.Numerical}) that provides some insights into the computational effort needed to implement our strategy and its potential cost benefit.

\section{Literature review} \label{S.LiteratureReview}

The idea of using shared mobility to deliver packages is still in its infancy, but there are already a few relevant studies worth mentioning.  Li et al.~\cite{li2014share} proposed a framework to use taxis to transport  people and deliver parcels at the same time, where the goal is to determine an optimal pick-up/drop-off sequence. They first consider a static version of the problem where all the people and parcel locations are known and formulate a mixed integer linear program (MILP) to find the optimal pick-up/drop-off sequence; then, they repeatedly solve their static formulation over small periods of time to obtain a more realistic dynamic scenario over a longer time window. In their dynamic approach, parcel locations are revealed at the beginning of the day while passenger locations are updated throughout the day, which means that a full schedule for picking up and delivering parcels cannot be computed upfront. Moreover, the dynamic version requires that the sequence is recomputed every time a passenger request appears.  

Qi et al.~\cite{qi2017shared} consider the problem of delivering packages from a warehouse to locations in a given region, and propose to first subdivide the region into a number of sub-regions, use a van system to distribute the packages among the centers of the sub-regions (referred to as ``terminals") , and then use shared mobility to do the last-mile delivery within each sub-region. The goal is to determine the optimal size (number) of the sub-regions, which is done using a continuous-approximation for the total average cost, and once the sub-regions are determined, the routes for delivering packages within each sub-region are computed by solving an open vehicle routing problem (OVRP).  Note that the optimization problem to identify the regions is meant to be solved only once and the OVRP is solved once per day, unlike the work in \cite{li2014share} where the optimization problem is solved many times every day. 

Perhaps the closest to our approach is the work done in Kafle et al.~\cite{kafle2017design}, where the authors consider using cyclists and pedestrians as crowdsources to help an in-house van system to do the last-leg delivery of parcels originating from a central warehouse. Their model assumes that parcel destinations are revealed at the beginning of the day, at which point cyclists and pedestrians are expected to submit bids to deliver them; then, once all bids have been received, the warehouse solves a mixed integer non-linear program to compute the optimal assignment, with any unassigned parcels being delivered by the in-house vans. Since solving exactly the mixed integer non-linear program is computationally very expensive, the authors propose a tabu search approximate algorithm instead. Compared to our model, we can think of the approach in \cite{kafle2017design} as looking at a multi-player game where the cyclists and pedestrians compete with each other with their bids and the warehouse optimizes over the entire set of bids, while in our setup there is only one player, the warehouse, that computes the optimal rewards for each package based on the expected response of the private drivers. In \cite{kafle2017design} both parcel destinations and bids need to be known at the beginning of the day while in our approach only package destinations need to be revealed while private driver requests remain stochastic.  In both \cite{kafle2017design} and our proposed framework the optimization problem is meant to be solved once every day.

It is worth reiterating that the methodology we propose to compute the payment (reward) offered to private drivers for each individual package is simple and flexible, since it only requires that we optimize a single-variable cost function at the beginning of the day and it allows private drivers to pick up bundles of packages, i.e., not only one package at a time.  As our main results show, this is enough to guarantee a lower total expected cost for the delivery of all packages under very natural cost assumptions. Our modeling approach could also become a building block for constructing improved online pricing models, where the rewards offered to private drivers change throughout the day as more packages are delivered and the allotted time decreases, which would very likely provide even better cost reductions at the expense of more computational effort.

\section{Model description} \label{S.Model}

In the first part of this section we model our last-mile delivery problem using private drivers as a \emph{discrete sequential packing} problem (DSP), and provide both exact and asymptotic results for the expected number of packages that can be delivered during the time window $[0,T]$.   In the second part of this section we describe a payment scheme for the private drivers based on this expectation and introduce a joint optimization problem to calculate the optimal incentive rate for minimizing the total expected cost.

\subsection{The expected number of packages that can be picked} \label{MainResult}

Throughout the paper we consider the problem of delivering $n$ packages over a time window of length $T$. The packages are assumed to have destinations spread over a region of a two-dimensional plane. This region is expected to be a population-dense area where package destinations are close to each other. 

During the period $[0, T]$ the distribution center receives requests from private drivers to deliver available packages. Any undelivered packages at time $T$ will then be delivered by the distribution center's vans/trucks. Moreover, the private drivers are allowed to take a bundle of packages with them, which we assume would have destinations in close proximity to each other. To model this proximity of the destinations, we first arrange all $n$ packages on a circle using the solution (exact or approximate) to the traveling salesman problem (TSP). The idea behind arranging the packages in this way is to reduce the complexity of dealing with the two-dimensional package locations by reorganizing them into one dimension. Note that doing so allows us to disregard the package destinations in the computation of the expected number of packages that will be picked up by private drivers, and we will argue that we can do this since  our pricing mechanism will make the profit rate for delivering each package equal for all packages.

From this point onwards, we can think of the packages as arranged on a circle with $n$ points, where point $i$ will be referred to as the ``location" of package $i$ (see Figure~\ref{F.Tour}). A bundle of size $k$ at location $i$ is the set consisting of the packages at locations $\{i, i+1, \dots, i+k-1\}$. Each location $i$ is associated to a marked Poisson process with rate $\lambda_i$, which represents the requests from private drivers to deliver a bundle whose first package is at location $i$. The marks of the Poisson process determine the size of the bundle that the driver would want to deliver, and is assumed to be distributed according to some distribution $F$, independent of the Poisson process and of any other marked Poisson processes at different locations. We assume that each driver has a preconceived idea of which package(s) they want to pick up, and that if at the time of the request one or more of the desired packages is no longer available, their request is rejected.  In order for all packages to be equally likely to be requested, we will assume that $\lambda_i = \lambda$ for all $1 \leq i \leq n$. Moreover, we assume that $\lambda$ remains constant during $[0, T]$, which would be the case if the pick-up window is chosen to coincide with private drivers' off-peak hours. Once a bundle request is accepted, the driver is also given the segment of the optimal TSP route for delivering all the packages in the bundle. 

Throughout the time interval $[0, t]$, the packages at the $n$ locations are picked up by private drivers, and our goal is to analyze the expected number of packages that can be delivered in this way over the interval $[0, t]$, which we denote by $\mathcal{C}(t,n,\lambda)$. Our analysis of $\mathcal{C}(t,n,\lambda)$ is based on the observation that once the first bundle, say of size $B$, is picked up, the remaining $n-B$ packages can be arranged on a line (see Figure~\ref{F.Delivery}). In fact, our main results for $\mathcal{C}(t,n,\lambda)$ are obtained by analyzing the expected number of packages that can be picked up during the interval $[0,t]$ when we start with $n$ packages arranged on a line, which we denote by $\mathcal{K}(t,n,\lambda)$. Figures~\ref{F.Tour} and \ref{F.Delivery} show an example of the pick-up process. We refer to this model as a {\em discrete sequential packing} (DSP) problem since it has strong connections to R\'enyi's parking/packing problem \cite{renyi1958one}.

The next subsection includes our main results for the computation of $\mathcal{C}(t,n,\lambda)$. We point out that the DSP formulation is based only on the following two assumptions:
\begin{itemize}
\item[i)] The marks of the Poisson process at location $i$, corresponding to bundle size requests and denoted $\{ B_j^{(i)} \}_{j \geq 1}$, are i.i.d.~with common distribution $F$ for all $1 \leq i \leq n$, and are independent of the Poisson processes at every location. 
\item[ii)] The arrival rates of the different Poisson processes, which correspond to the arrival rates for requests from drivers to pick up  bundles at location $i$, are all equal, i.e., $\lambda_i = \lambda$ for all~$1 \leq i \leq n$.
\end{itemize}

\begin{figure}[H]
\begin{center}
  \includegraphics[width=7cm]{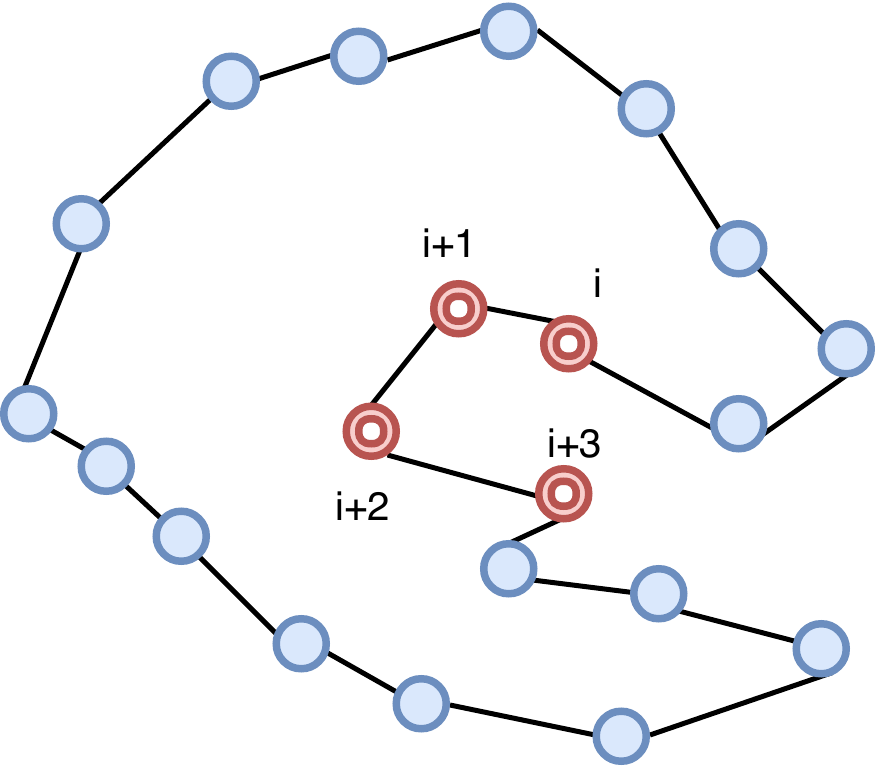}
  \caption{TSP tour. The picture illustrates a case with $n = 21$ packages, where the first bundle to be requested is of size four (double circular packages). After the first bundle is picked up, the remaining packages are arranged on a line.}
\label{F.Tour}
  \end{center}
\end{figure}

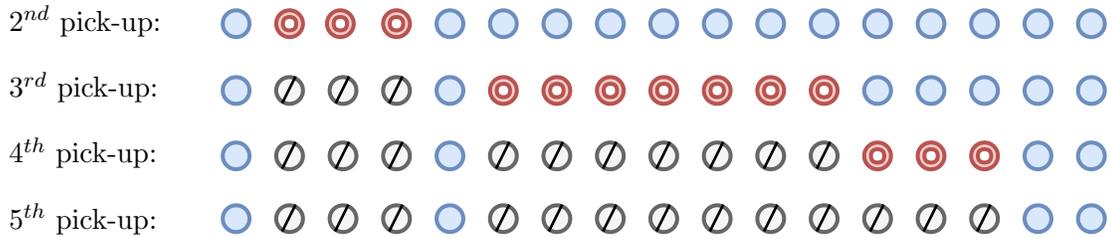
\begin{figure}[H]
\begin{center}
\begin{picture}(400,80)
\put(80,0){\includegraphics[scale= 0.7, bb = 0 0 680 170, clip]{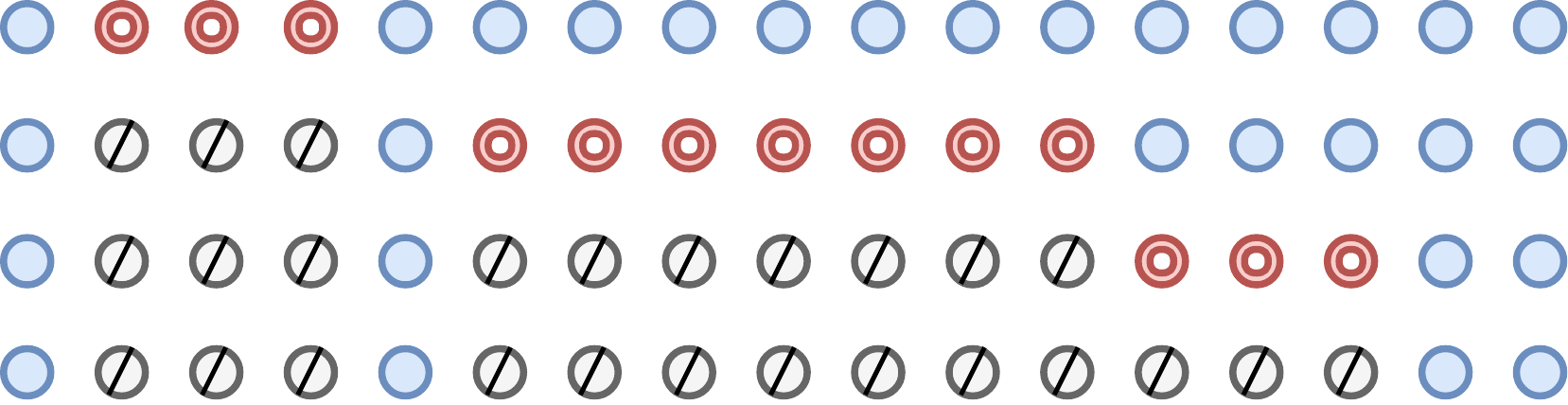}}
\put(0,78){2$^{nd}$ pick-up:}
\put(0,53){3$^{rd}$ pick-up:}
\put(0,28){4$^{th}$ pick-up:}
\put(0,3){5$^{th}$ pick-up:}
\end{picture}
  \end{center}
  \caption{Depiction of the package pick-up process. Double circular packages correspond to accepted bundle requests; circular packages are yet to be picked up; circle-with-a-line packages are already taken.}
\label{F.Delivery}
\end{figure}

As mentioned above, the first of these assumptions is justified in our package delivery context by the TSP route, which in all our numerical experiments (see Section~\ref{S.Numerical}) seems to provide solutions where adjacent packages along the route are typically small, making all bundles of the same size (approximately) equally desirable. The second assumption, which ensures that the arrival rate for requests at all $n$ locations are the same, will be justified by the pricing mechanism described in Section~\ref{S.Pricing}, along with the assumption that our business targets on a population-dense area and packages are close to each other. In particular, we propose a pricing scheme that incorporates the distance from the distribution center to each destination and the neighboring distance between packages along the TSP route, among other parameters, in such a way that the profit for delivering any of the $n$ packages is essentially the same. Since intuitively the total number of packages that need to be delivered and the supply of drivers in a given region are both proportional to its size, we can reasonably assume that $\lambda$ does not depend on $n$. 

To conclude this section, we give a quick overview of the related literature. As mentioned earlier, our model has strong connections to R\'enyi's parking/packing problem~\cite{renyi1958one}, which examines the problem of filling up an interval $I = [0, n]$ with subintervals of length one whose left endpoints are uniformly chosen at random, and computes the asymptotic proportion of filled space (as $n \to \infty$). A discretized version of this problem was introduced by Page~\cite{page1959distribution}, by considering $n$ points instead of the interval $I$ and replacing the subintervals of unit length with pairs of adjacent points. The main results in \cite{page1959distribution} include the mean and other moments of the distribution of the remaining isolated points, and the setup for the problem corresponds to choosing $P(B = 2) = 1$ in the description of our model. Page's work was later generalized, both theoretically and numerically, to the case $P(B = m) = 1$ in~\cite{clay2016renyi, mackenzie1962sequential, pinsky2014problems}. We point out that none of the references mentioned above require the use of a ``time" component for the random selection of subintervals/points, since only the distribution of unfilled space is of interest. In this regard, our work is more closely related to that of~\cite{coffman1998packing}, which assumes that the times at which the subintervals are chosen in R\'enyi's parking/packing problem occur according to a Poisson process, in which case one can focus on the proportion of filled space by time $t$. The results in~\cite{coffman1998packing} include the asymptotic proportion of filled space by time $t$, as $n \to \infty$, and a corresponding central limit theorem. 

For implementation purposes, we also refer the reader to \cite{applegate2011traveling,cook2011pursuit,hoffman2013traveling,lawler1985traveling} for further details on the TSP solution and existing algorithmic approaches for its computation \cite{aarts1988quantitative, nagata2006new, potvin1996genetic}.


 \subsubsection{Results for the discrete sequential packing problem} \label{SS.MainResultDSP}
 
 Our main results for $\mathcal{C}(t,n,\lambda)$, the expected number of packages that can be delivered during a time period $[0,t]$,  include an exact calculation obtained by solving a differential equation, and its asymptotic behavior as the number of packages goes to infinity. Throughout this section, the arrival rate $\lambda_i$ at each location is assumed to be the same for all locations, say $\lambda$, and the distribution of the bundle sizes, denoted $F$, is assumed to have support on $\{1, 2, \dots, m\}$ for some fixed $m \in\mathbb{N}_+$. 
 
The first result gives a differential equation satisfied by $\mathcal{K}(t,n,\lambda)$, which we recall is the expected number of packages that can be picked up by time $t$ when $n$ packages are arranged on a line. 
 
 \begin{theorem}\label{ode}
$\mathcal{K}(t,n,\lambda)$ satisfies the following recursive differential equation
\begin{equation*}
\frac{1}{\lambda}\frac{\partial\mathcal{K}(t,n,\lambda)}{\partial t}=-\sum_{j=1}^n F(j) \mathcal{K}(t,n,\lambda)+2\sum_{i=1}^{n-1}F(n-i)\mathcal{K}(t,i,\lambda)+\sum_{i=1}^n f(i)i(n+1-i)
\end{equation*}
with boundary condition $\mathcal{K}(0,n,\lambda)=0$. Moreover, $\mathcal{R}(t,n,\lambda) = n - \mathcal{K}(t,n,\lambda)$ satisfies the recursive differential equation:
\begin{equation} \label{OdeEq}
\frac{1}{\lambda}\frac{\partial\mathcal{R}(t,n,\lambda)}{\partial t} = -\sum_{i=1}^n F(i) \mathcal{R}(t,n,\lambda )  + 2\sum_{i=1}^{n-1}F(n-i) \mathcal{R}(t,i,\lambda)
\end{equation}
with boundary condition $\mathcal{R}(0,n,\lambda)=n$.
\end{theorem}

The corresponding solution is given by the following theorem. 

\begin{proposition}\label{recursive}
The solution to equation~\eqref{OdeEq} is given by
\begin{equation}
\mathcal{R}(t,n,\lambda)=
\left\{
\begin{array}{l@{\;\;}l}
\displaystyle \sum_{i=1}^n \gamma_{n,i} \, e^{-\lambda\sum_{j=1}^i F(j) t}&\text{ if } \displaystyle \sum_{j=1}^n F(j)>0\\
n& \text{ otherwise, }
\end{array}
\right.
\end{equation}
where the constants $\gamma_{n,i}$ can be computed recursively according to
$$\gamma_{n,n}= n -\sum_{i=1}^{n-1} \gamma_{n,i} \quad\text{and}\quad\gamma_{n,i}= 2 \cdot \frac{\sum_{j=1}^{n-i}F(j) \gamma_{n-j,i}}{\sum_{k=i+1}^{n}F(k)}, \quad 1 \leq i < n,$$
with boundary value $\gamma_{i,j}=1$ for all $i,j$ such that $F(i)=0$ and $j\leq i$. Furthermore, $\mathcal{C}(t,n,\lambda)$ is given by
$$\mathcal{C}(t,n,\lambda)=n- \sum_{i=1}^{n-1} \tilde\gamma_{n,i}e^{-\lambda\sum_{j=1}^i F(j) t}-\tilde \gamma_{n,n}e^{-\lambda nt},\text{ for  }n\geq m$$
where 
$$\tilde \gamma_{n,i}=\sum_{k=1}^{n-i}f(k)\frac{\gamma_{n-k,i}}{1-\frac{1}{n}\sum_{j=1}^i F(j)}\quad\text{ and }\quad\tilde \gamma_{n,n}=n-\sum_{i=1}^{n-1}\tilde \gamma_{n,i}.$$
\end{proposition}

Since in the setting of delivering packages that we consider here it is safe to assume that the number of packages $n$ will be large, we can compute the asymptotic proportion of packages that can be delivered by private drivers during the time interval $[0, t]$ as $n \to \infty$. The corresponding result is given below, and can be used when $n$ is large to simplify the computation of the cost function in Section~\ref{S.Pricing}.

\begin{theorem}\label{alpha}
Define 
$$\alpha(t,\lambda) = \lim_{n \to \infty} \frac{\mathcal{K}(t,n,\lambda)}{n}$$
to be the asymptotic proportion of delivered packages during the time interval $[0,t]$, with arrival rate $\lambda$. Then,
$$ \lim_{n \to \infty} \frac{\mathcal{C}(t,n, \lambda)}{n} = \alpha(t,\lambda) $$
and
\begin{align} 
\alpha(t,\lambda) &= 1 -  \int_{e^{-\lambda t}}^1 e^{2(\varphi(u)- \varphi(1)) } \hat q(t+(\ln u)/\lambda, u) \, du  -  (m-(m-1)e^{-\lambda t}) e^{2(\varphi(e^{-\lambda t}) - \varphi(1)) - \lambda t(m-\varphi'(1))} ,  \label{limit_alpha}
\end{align}
where 
$$\varphi(y)=\sum_{i=1}^{m-1}\frac{\overline{F}(i)}{i}y^i,$$
$\overline{F}(i) = 1 - F(i)$, $\mathcal{R}(t,n,\lambda) = n - \mathcal{K}(t,n,\lambda)$ and
$$\hat q(s,v) := 2 v^{m-\varphi'(1)-1} (1-v) \sum_{i=1}^{m-1} \mathcal{R}(s, i) - 2(1-v)^2 \sum_{i=1}^{m-1} \mathcal{R}(s, i) \sum_{j=m-i}^{m-1} \overline{F}(j) v^{i+j - \varphi'(1)-1}.$$
\end{theorem}

Note that to calculate $\alpha(t,\lambda)$, we only need to compute $\{ \gamma_{i,j} : 1 \leq j \leq i\}$ for $i<m$, since only the terms $\mathcal{R}(t + (\ln u)/\lambda, i, \lambda )$ for $1 \leq i \leq m-1$ appear in the expression for $\hat q(t+(\ln u)/\lambda, u)$.  Moreover, whenever $P(B = 1) > 0$ we have that $\lim_{t \to\infty} \mathcal{R}(t,i) = 0$ for all $i \geq 1$, which yields 
$$\lim_{t \to \infty} \hat q(t+(\ln u)/\lambda, u) = 0$$
for all $u \in [0,1)$, and therefore,
$$\lim_{t \to \infty} \alpha(t,\lambda) =1,$$
as expected.  Note that Theorem~\ref{alpha} can also be used to compute the limiting proportion of remaining packages when $P(B=1)=0$, in which case the quantity of interest is the number of remaining packages at time $t = \infty$.  For the special case $P(B = m) = 1$ $(m > 1)$ Pinsky \cite{pinsky2014problems}  obtained the formula
$$\alpha(\infty,\lambda) = m\exp \left(-2\sum_{j=1}^{m-1}\frac{1}{j} \right)\int_0^1 \exp\left(2\sum_{j=1}^{m-1}\frac{u^j}{j} \right) du = m\int_0^1 e^{2\varphi(u)-\varphi(1)} du$$
which can be derived from Theorem~\ref{alpha} by noting that $\varphi'(1) = m -1$ and
\begin{align*}
\lim_{t \to \infty} \hat q(t+(\ln u)/\lambda, u) &= 2(1-u) \sum_{i=0}^{m-1} i \left(1 - (1-u) \sum_{j=m-i}^{m-1} u^{i+j-m} \right)\\
&=2(1-u)\sum_{i=0}^{m-1}iy^i = 2(u\varphi'(u)-(m-1)u^m)\\
&= 2 \varphi'(u) (mu - (m-1)u^2 ) - 2(m-1)u 
\end{align*}
for all $\lambda > 0$, and therefore, 
\begin{align*}
\lim_{t \to \infty} \alpha(t,\lambda) &= 1 - \int_0^1 e^{2(\varphi(u)- \varphi(1)) } \lim_{t \to \infty} \hat q(t+(\ln u)/\lambda, u) \, du \\
&= 1 - \int_0^1 e^{2(\varphi(u) - \varphi(1))} 2\varphi'(y) (mu - (m-1)u^2) dy + 2(m-1) \int_0^1 e^{2(\varphi(u) - \varphi(1))}  u \, du \\
&=   \int_0^1 e^{2(\varphi(u) - \varphi(1))} (m - 2(m-1)u) \, du + 2(m-1) \int_0^1 e^{2(\varphi(u) - \varphi(1))}  y \, du  \\
&= m\int_0^1 e^{2(\varphi(u) - \varphi(1))} du. 
\end{align*}

In addition to the expression for $\alpha(t,\lambda)$ in Theorem~\ref{alpha}, we also provide the corresponding rate of convergence of $\frac{\mathcal{C}(t,n,\lambda)}{n}$ to $\alpha(t,\lambda)$. To the best of our knowledge, this is the first result regarding the rate of convergence for the limiting proportion of packages picked up by time $t$. Although we do not include the details of the proof in this paper, a similar set of arguments as those used in the proof of Theorem~\ref{convergence} also yield the result for $t = \infty$. Throughout the paper we use $f(x) = O(g(x))$ as $x \to \infty$ if $\limsup_{x \to \infty} |f(x)/g(x)| < \infty$.

\begin{theorem}\label{convergence}
For any fixed $t$,
$$\left|\frac{\mathcal{C}(t,n,\lambda)}{n}-\alpha(t,\lambda )\right|=O(n^{-1}), \quad \text{ as } n\rightarrow\infty .$$
\end{theorem}

We now move on to the pricing part of our model.

\subsection{Computing the reward for each package} \label{S.Pricing}

The overall goal of the proposed framework is to provide a pricing strategy for delivering $n$ different packages using a combination of private drivers and in-house delivery vans. Section \ref{MainResult} provided analytical results for the expected number of packages delivered during time $[0,T]$ as a function of the arrival rate $\lambda$, under the assumption that all packages are equally desirable ($\lambda$ is the same for all $n$ locations). As mentioned earlier, it is through the pricing mechanism that we will justify the modeling assumption on $\lambda$, since we would naturally expect that packages with remote destinations would receive fewer requests. Our payment scheme is based on the idea that the amount of money that a driver can make per unit of time should be the same for all packages, and we accomplish this by separating the costs associated to the destination of each package from those of a common ``incentive rate". The package specific costs will take into account factors such as the distance between the destination and the warehouse (long-haul distance) and the distance between neighboring packages with respect to the TSP route (local distance).  Once we have provided an expression for the cost of delivering packages through the use of private drivers, we will need to estimate the cost of delivering the remaining packages using in-house vans. The detailed description of our pricing for the delivery using private drivers is given in Section~\ref{SS.private}, and the corresponding vans' cost is given in Section~\ref{SS.vans}.

\subsubsection{Rewards for private drivers} \label{SS.private}

For the delivery process, assume that both private drivers and in-house vans must pick up the packages they will be delivering from the distribution center.  At the beginning of the day, the destinations of the $n$ packages to be delivered that day are revealed and an optimal TSP tour is computed. Let $\{{\bf x}_1,..., {\bf x}_n\} \subseteq \mathbb{R}^2$ denote the destinations of the $n$ packages, where their indexes correspond to their ``locations" within the TSP route. The distribution center is assumed to be at the origin. We denote by $\hat{B}^{(i)}$ the size of the bundle of the first request at location $i$ to be accepted. Note that the distribution of $\hat{B}^{(i)}$ is not $F$ since the acceptance depends on the configuration of available packages at the time of the request. 

The overall cost for private drivers to deliver packages consists of transportation costs and opportunity costs, since drivers can also choose to work for Uber-like companies or take other jobs instead. We use the following quantities in our cost estimation; $d({\bf x}, {\bf y})$ denotes the distance on $\mathbb{R}^2$:

\begin{itemize} 
\item $r_i = d({\bf x}_i, \boldsymbol{0})$ denotes the distance from the depot to the destination of package $i$ (the long-haul distance).
\item $d_i = (d({\bf x}_{i-1}, {\bf x}_i) + d({\bf x}_i, {\bf x}_{i+1}))/2$ is the average distance between the destinations of packages $i-1$ and $i$, and $i$ and $i+1$, for $2\leq i\leq n-1$, $ d_1=(d({\bf x}_{n}, {\bf x}_1) + d({\bf x}_1, {\bf x}_{2}))/2$, $ d_n=(d({\bf x}_{n-1}, {\bf x}_n) + d({\bf x}_n, {\bf x}_{1}))/2$  (the local distance).
\item $\zeta_P$ is the per-mile transportation cost.
\item $h_P$ is the opportunity cost per unit of time.
\item $\tau_P$ is the end-point delivery time.
\item $v_P$ is the average speed of private cars.
\end{itemize}

The decision variable in our pricing model will be an incentive rate $z$ that each driver will receive in addition to the opportunity cost, i.e., the total payment rate that a driver receives per unit of time is $h_P+z$. Note that the only quantities in the cost that depend on the geographic location of the package destinations are the $\{r_i\}$ and the $\{d_i\}$. The $\{ r_i\}$ can be computed as soon as the destinations $\{ {\bf x}_i\}$ are revealed, while the $\{d_i\}$ are determined by the TSP route.

The traveling distance associated with a bundle of packages at locations $\{i, i+1,\dots, i+k-1\}$ is $r_i+\sum_{j=i}^{i+k-2}d({\bf{x}}_j,{\bf{x}}_{j+1})$. Thus, the price set for the bundle should be 
 \begin{equation}\label{price.bundle}
\text{price}_1 = \zeta_P \left(r_i+\sum_{j=i}^{i+k-2}d({\bf{x}}_j,{\bf{x}}_{j+1})\right)+(h_P+z)\left(r_i/v_P +\sum_{j=i}^{i+k-2}d({\bf{x}}_j,{\bf{x}}_{j+1})/v_P+k\tau_P\right).
\end{equation}
However, since the number of possible bundles increases geometrically with the number of total packages, it is computationally expensive to set a price for every possible bundle. Therefore, we consider instead a pricing scheme for each individual package, regardless of which bundle it will be included in. To derive this price, suppose that package $j$ is delivered as part of a bundle of size $k$, and start by prorating the long-haul cost among all the $k$ packages, and separate the contribution of package $j$ to the local distance. To incorporate into the pricing longer neighboring distances between adjacent packages in the TSP tour, we determine the contribution of package $j$ to the local distance to be the average of the distances to both the neighbor to the left and the neighbor to the right, i.e., $d_j = (d({\bf x}_{j-1}, {\bf x}_j) + d({\bf x}_j, {\bf x}_{j+1}))/2$. We also argue that the long-haul cost of package $j$ is approximately the same as that of other packages in the same bundle, and therefore if $j$ is part of a bundle accepted at location $i$, then $r_j \approx r_i$. We then propose the payment reward for package $j$ to be:
\begin{equation} \label{eq:Prorated}
\zeta_P \left( \frac{r_j}{k}+d_j \right)+(h_P +z) \left( \frac{r_j}{kv_P} + \frac{d_j}{v_P}+\tau_P \right).
\end{equation}
Using \eqref{eq:Prorated} we obtain a price for a bundle of size $k$ accepted at location $i$ of the form:
 \begin{align}\label{price.package}
\text{price}_2 &= \zeta_P \left( \frac{1}{k} \sum_{j=i}^{i+k-1} r_j+\sum_{j=i}^{i+k-1} d_j \right) +(h_P+z)\left( \frac{1}{k v_P} \sum_{j=i}^{i+k-1}r_j +\sum_{j=i}^{i+k-1}d_j/v_P+k\tau_P \right)  \\
&= \text{price}_1 + \left( \zeta_P + \frac{h_P+z}{v_P} \right) \left( \frac{1}{k} \sum_{j=i}^{i+k-1} r_j - r_i \right) \notag \\
&\hspace{5mm} + \frac{1}{2}(d({\bf x}_{i-1}, {\bf x}_{i})+ d({\bf x}_{i+k-2}, {\bf x}_{i+k-1}))\left(\zeta+\frac{h_P +z}{v_P}\right) .  \notag
 \end{align}
Note that the difference between \eqref{price.bundle} and \eqref{price.package} is small whenever adjacent packages in the TSP tour are small, which we expect to be true for large $n$. To the best of our knowledge, there are no theoretical results about the distribution of the neighboring distance in an optimal TSP route, although our numerical experiments (see Section~\ref{SS.TSPneighbor}) do indeed suggest that this will be the case.

To obtain our proposed expression for the cost to deliver each of the $n$ packages we also need to take into account that package $j$ could be delivered as part of a number of different bundles, e.g., it would be delivered in bundle $\{j, \dots, j+\hat B^{(j)}-1 \}$ or it could be contained in a bundle of the form $\{i, \dots, j, \dots, i+\hat B^{(i)}-1\}$ for some $i < j$. Since the exact computation of the distribution of the size of the bundle containing $j$ is too complex, we approximate it with $E[B]$ to obtain the following price for package $j$:
\begin{equation} \label{eq:PriceJ}
p_j := \zeta_P \left( \frac{r_j}{E[B]}+d_j \right)+(h_P+z) \left( \frac{r_j}{E[B] v_P} + \frac{d_j}{v_P}+\tau_P \right).
\end{equation}
Using the same type of arguments, we estimate the time required to deliver package $j$ to be:
$$t_j := \frac{r_j}{E[B]v_P} + \frac{d_j}{v_P} + \tau_P.$$

Note that by deriving our pricing mechanism the way we did we have made the profit for the drivers to be the same regardless of which package(s) they choose to deliver. This profit is determined in our pricing scheme by the incentive rate $z$, which is the same for all $n$ packages and is linear in the number of packages that a driver chooses to deliver. Moreover, the incentive rate will be used to control the arrival rate for requests in our calculations from Section~\ref{MainResult} by setting $\lambda = \lambda(z)$ to be an non-decreasing function. It remains to set up an optimization problem for determining the best incentive rate to use. 

To this end, start by noting that the sum of the prices for all $n$ packages after their destinations are revealed satisfies:
\begin{align*}
\sum_{i=1}^n p_i = \left(\zeta_P +\frac{h_P +z}{v_P}\right)\left( \frac{1}{E[B]} \sum_{i=1}^n r_i+L(\text{TSP}({\bf x}^{(n)}))\right)+n(h_P+z)\tau_P,
\end{align*}
where $L(\text{TSP}({\bf x}^{(n)}))$ is the length of an optimal TSP route for points with destinations ${\bf x}^{(n)} :=\{{\bf x}_1,{\bf x}_2,...,{\bf x}_n\}$. Since the probability that package $i$ will be picked up before time $T$ given incentive rate $z$ is $
\mathcal{C}(T,n,\lambda(z))/n$ (recall that packages are arranged on a circle, and are therefore undistinguishable), the expected payment for private drivers (conditional on ${\bf x}^{(n)}$) is 
\begin{equation} \label{eq:CostPrivateDrivers}
\frac{\mathcal{C}(T,n,\lambda(z))}{n}\left(\zeta_P+\frac{h_P+z}{v_P}\right)\left(\frac{1}{E[B]} \sum_{i=1}^n r_i+L(\text{TSP}({\bf x}^{(n)}))\right)+\mathcal{C}(T,n,\lambda(z))(h_P+z) \tau_P.
\end{equation}

It remains to compute the cost associated to delivering the remaining packages using the in-house van service, which we do in the following section.

\subsubsection{Van's cost to deliver leftover packages} \label{SS.vans}

After time $T$, all leftover packages will be delivered by vans owned by the distribution center. The delivery route for a van is designed by an optimal CVRP, which is also NP-hard. For more detailed information about the algorithms that can be used to solve the CVRP problem, we refer readers to~\cite{fukasawa2006robust,golden2008vehicle,toth2002vehicle}. 

The van's operating cost includes the per-mile cost for vans and time-based wages for the drivers. The per-mile cost includes the cost of fuel, maintenance, repairs, depreciation, etc., and is denoted by $\zeta_V$. To compute the time-based wages for the drivers note that the time they spend delivering packages includes the time driving along an optimal CVRP route and the time on the end-point delivery. Let $h_V, v_V,\tau_V$ denote the drivers' payment rate, the vans' average speed, and the end-point delivery time, respectively. Then, the total cost for delivering $k$ packages using vans is 
$$\left(\zeta_V + \frac{h_V}{v_V}\right) L(\text{CVRP}({\bf y}^{(k)}))+k h_V\tau_V,$$
where $L(\text{CVRP}({\bf y}^{(k)}))$ is the length of capacitated vehicle routing through the points ${\bf y}^{(k)}=\{{\bf y}_1,{\bf y}_2,..,{\bf y}_k\}$.

In the context of our problem, the number of packages that will need to be delivered after time $T$ is random, and so are their destinations, which we will denote by ${\bf Y}^{(n)}(\lambda) = \{ {\bf Y}_1, \dots, {\bf Y}_{n-\tilde N(T,n,\lambda)} \}$, where $\tilde N(T,n,\lambda)$ is the number of packages that can be delivered during $[0,T]$ when we start with $n$ packages arranged on a circle.  It follows that the expected cost to deliver the remaining packages, conditionally on the destinations ${\bf x}^{(n)}$, is given by
\begin{equation} \label{eq:IdealCostRemainingPackages}
\left(\zeta_V +\frac{h_V}{v_V}\right) \mathbb{E}_n\left[ L(\text{CVRP}({\bf Y}^{(n)}(\lambda) )) \right] + \mathbb{E}_n[ n - \tilde N(T,n,\lambda(z)) ]  h_V \tau_V,
\end{equation}
where $\mathbb{E}_n[ \, \cdot \, ] = E[ \, \cdot \, | {\bf x}^{(n)}]$. However the computation of the conditional expectation appearing in \eqref{eq:IdealCostRemainingPackages} would be hard to do, both from the point of view of the distribution of ${\bf Y}^{(n)}(\lambda)$ and of the length of the CVRP itself. Therefore, we use a continuous approximation result (see \cite{daganzo1984distance}) that is known to work well when the capacity of the vans is significantly smaller than the number of packages that need to be delivered ($V = o(\sqrt{k})$, where $k$ is the number of packages), which yields
$$L(\text{CVRP}({\bf y}^{(k)}) )  \approx  \frac{2k}{V}  \overline{r}(k) + \beta_\text{VRP} \sqrt{k A},$$
whenever the destinations $\{ {\bf y}_1, \dots, {\bf y}_k\}$ are uniformly distributed over a region with area $A$. The constant $ \beta_\text{VRP}$ is estimated to be $0.82$ when using the $L_1$ distance (see Appendix~A in 
 \cite{daganzo2005logistics}). The paper \cite{snyder1995equidistribution} proves a strong result stating that in an asymptotic sense the worst-case point sets are uniformly distributed.

Replacing $k$ with $\mathbb{E}_n[ n - \tilde N(T,n,\lambda(z)) ] = n - \mathcal{C}(T,n,\lambda(z))$ and $\overline{r}(k)$ with $\overline{r}(n) = n^{-1} \sum_{i=1}^n r_i$ in the previous inequality gives the more tractable cost:
\begin{align} \label{eq:CostRemainingPackages}
&\left(\zeta_V+\frac{h_V}{v_V}\right) \left(  \frac{2(n - \mathcal{C}(T,n,\lambda(z)))  \overline{r}(n)}{V}  + \beta_\text{VRP} \sqrt{(n- \mathcal{C}(T,n,\lambda(z))) A} \right) + (n - \mathcal{C}(T,n,\lambda(z))) h_V \tau_V.
\end{align}

\subsubsection{Joint optimization problem}

Putting together our cost estimations for both the private drivers \eqref{eq:CostPrivateDrivers} and that of delivering the remaining packages using the in-house vans \eqref{eq:CostRemainingPackages}, we obtain the following expected cost function for delivering $n$ packages with destinations ${\bf x}^{(n)} = \{ {\bf x}_1, \dots, {\bf x}_n\}$ using the incentive rate $z$:
\begin{align*}
\text{Cost}(z; {\bf x}^{(n)}) &:=  \frac{\mathcal{C}(T,n,\lambda(z))}{n}\left(\zeta_P +\frac{h_P+z}{v_P}\right)\left(\frac{n \overline{r}(n) }{E[B]} + L(\text{TSP}({\bf x}^{(n)}))\right) \\
&\hspace{5mm} +\mathcal{C}(T,n,\lambda(z))(h_P+z) \tau_P  +(n - \mathcal{C}(T,n,\lambda(z))) h_V \tau_V \\
&\hspace{5mm}  + \left(\zeta_V+\frac{h_V}{v_V}\right) \left(  \frac{2(n - \mathcal{C}(T,n,\lambda(z))) \overline{r}(n)}{V}   + \beta_\text{VRP} \sqrt{(n- \mathcal{C}(T,n,\lambda(z))) A} \right),
\end{align*}
where $V$ is the capacity of the vans used by the distribution center, $A$ is the area of the plane where the package destinations ${\bf x}^{(n)}$ lie, $\overline{r}(n) = n^{-1} \sum_{i=1}^k r_i$, and $\beta_\text{VRP}$ can be taken to be equal to 0.82 when using the $L_1$ metric.  We will now argue that it suffices to minimize $\text{Cost}(z; {\bf x}^{(n)})$ over a bounded interval. The arrival rate $\lambda(z)$ can be taken to be any non-negative monotone non-decreasing and differentiable almost everywhere function on the real line, e.g., linear or piecewise linear. 

Note that the distribution center will be paying private drivers $\zeta_P + (h_P+z)/v_P$ per mile travelled, plus $(h_P + z)\tau_P$ per package delivered, while it will pay its van drivers $\zeta_V + h_V/v_V$ per mile travelled and $h_V \tau_V$ per package delivered. Hence, in order for a strategy using private drivers to even make sense, we would need at least one of the following to hold:
$$\zeta_V + \frac{h_V}{v_V} > \zeta_P + \frac{h_P + z}{v_P} \qquad \text{or} \qquad h_V \tau_V > (h_P + z) \tau_P.$$
In other words,
$$z \leq  \max\left\{  \left( \zeta_V + \frac{h_V}{v_V} - \zeta_P \right) v_P , \, \frac{h_V\tau_V}{\tau_P}  \right\} - h_P. $$
On the other hand, since the payment rate to private drivers (not including the per-mile transportation cost $\zeta_P$) must be nonnegative, we have that $h_P + z \geq 0$, which implies that 
$$z \geq - h_P.$$

In view of the above, we propose to compute the optimal incentive rate $z^*$ by solving:
\begin{equation} \label{eq:Optimization}
\min_{-h_P \leq z \leq  \max\left\{  \left( \zeta_V + \frac{h_V}{v_V} - \zeta_P \right) v_P , \, \frac{h_V\tau_V}{\tau_P}  \right\} - h_P} \text{Cost}(z; {\bf x}^{(n)}).
\end{equation}
 The optimization problem is meant to be solved at the beginning of the day, once the destinations ${\bf x}^{(n)}$ are revealed. Moreover, in view of Theorems~\ref{alpha} and \ref{convergence}, when the exact computation of $\mathcal{C}(T,n,\lambda(z))$ is too costly, we can approximate it with $n\alpha(T,\lambda(z))$.  Once the optimal $z^*$ has been found the reward offered to private drivers for delivering package $i$ is $p_i$, as given by \eqref{eq:PriceJ}.

We point out that since $\mathcal{C}(T,n,\lambda(z))$ is infinitely differentiable in $z \in \mathbb{R}$ whenever $\lambda(z)$ is (see Proposition~\ref{recursive}), that solving the minimization problem in \eqref{eq:Optimization} can be done very efficiently. The problem is to be solved on a daily basis as soon as the destinations ${\bf x}^{(n)} = \{{\bf x}_1, \dots, {\bf x}_n\}$ are revealed.

We conclude the main results by providing sufficient conditions under which the expected cost to deliver $n$ packages with destinations ${\bf x}^{(n)} = \{ {\bf x}_1, \dots, {\bf x}_n\}$ using a combination of private drivers and in-house vans is smaller than that of using only in-house vans. For this comparison we use the more precise version of the expected cost for the private drivers strategy:
\begin{align*}
\text{Cost}_{P}(z; {\bf x}^{(n)}) &:= \frac{\mathcal{C}(T,n,\lambda(z))}{n}\left(\zeta_P+\frac{h_P+z}{v_P}\right)\left(\frac{n \overline{r}(n) }{E[B]} +L(\text{TSP}({\bf x}^{(n)}))\right) \\
&\hspace{5mm} +\mathcal{C}(T,n,\lambda(z))(h_P+z) \tau_P +  (n - \mathcal{C}(T,n,\lambda(z))  h_V \tau_V\\
&\hspace{5mm} + \left(\zeta_V+\frac{h_V}{v_V}\right) \mathbb{E}_n\left[ L(\text{CVRP}({\bf Y}^{(n)}(\lambda(z)) )) \right] ,
\end{align*}
and compare it to the expected cost of the van-only strategy:  
$$\text{Cost}_{V}({\bf x}^{(n)}) := \left( \zeta_V + \frac{h_V}{v_V} \right)  L(\text{CVRP}({\bf x}^{(n)}) )) + n h_V \tau_V.$$

\begin{lemma} \label{L.CompareStrategies}
Suppose the destinations ${\bf x}^{(n)} = \{ {\bf x}_1, \dots, {\bf x}_n\}$ are contained in a compact region $R \subseteq \mathbb{R}^2$ and are such that the limit $r^* = \lim_{n \to \infty} \overline{r}(n)$ exists. Then, for any $z \in \mathbb{R}$ and $\lambda = \lambda(z)$ we have 
\begin{align*}
&\limsup_{n \to \infty} \frac{1}{n} \left( \text{Cost}_{P}(z; {\bf x}^{(n)}) - \text{Cost}_V({\bf x}^{(n)}) \right)  \\
&\leq - \alpha(T,\lambda) \left(  \left(\zeta_V+\frac{h_V}{v_V}\right) \frac{2r^*}{V} -  \left(\zeta_P+\frac{h_P}{v_P}\right) \frac{r^*}{E[B]}  - (h_P\tau_P - h_V\tau_V) - z \left( \frac{r^*}{v_P E[B]} + \tau_P   \right)  \right)  \quad \text{a.s.}
\end{align*}
 Moreover, whenever
$$\left(\zeta_V+\frac{h_V}{v_V}\right) \frac{2r^*}{V} -  \left(\zeta_P+\frac{h_P}{v_P}\right) \frac{r^*}{E[B]}  - (h_P\tau_P - h_V\tau_V) >  0,$$
there exists a $z \in \mathbb{R}$ for which the upper bound for the limit superior is strictly negative, i.e., for which the private drivers strategy is better than the van-only strategy. 
\end{lemma}

\section{Proofs} \label{S.Proofs}

In this section, we give all the proofs of the theorems in Section~\ref{SS.MainResultDSP} and Section~\ref{S.Pricing}. The analysis of $\mathcal{C}(t,n,\lambda)$ is based on the observation that once the first package is picked up (which is guaranteed to be accepted whenever $n \geq m$), the remaining packages can be arranged on a line. Therefore, the proofs of all our results are based on the analysis of $\mathcal{K}(t,n,\lambda)$, the expected number of packages that can be picked up by private drivers during the interval $[0, t]$, when there are $n$ packages arranged on a line. Moreover, we point out that if $B$ has distribution $F$ and $T^*_n$ denotes the time of the first request when we have $n$  packages arranged on a circle, then 
$$\mathcal{C}(t,n,\lambda) = E\left[ \left( B + \mathcal{K}(t- T_n^*, n - B) \right) 1(T_n^* \leq t) \right],$$
with $T_n^*$ exponentially distributed with rate $\lambda n$. 

Throughout this section we simplify the notation by omitting $\lambda$ from $\mathcal{C}(t,n,\lambda)$, $\mathcal{K}(t,n,\lambda)$, $\mathcal{R}(t,n,\lambda)$, etc., and simply write $\mathcal{C}(t,n)$, $\mathcal{K}(t,n)$, $\mathcal{R}(t,n)$; all the proofs in this section are valid for any fixed $\lambda > 0$. We also use $f(k) = P(B =k)$ to denote the bundle size probability mass function.  The first proof corresponds to Theorem~\ref{ode}, which gives a differential equation for $\mathcal{K}(t,n)$ and $\mathcal{R}(t,n) = n - \mathcal{K}(t,n)$. 

\bigskip

\begin{proof} [Proof of Theorem \ref{ode}]
Define $N(t, n)$ to be the number of packages that can be delivered over the period $[0, t]$ when packages are arranged on a line, i.e., $\mathcal{K}(t, n) = E[ N(t,n)]$. 

Looking at the first $\Delta $ units of time after time $t$, we obtain
\begin{align}\label{eq1}
N(t+\Delta ,n)&=N(t+\Delta ,n)1(\text{no drivers arrive during }[t, t+\Delta])\notag\\
&\hspace{5mm}+\sum_{w=1}^m N(t+\Delta,n)1(\text{bundle of size $w$ arrives during }[t, t+ \Delta]) \notag \\
&\hspace{5mm} + N(t+\Delta, n) 1(\text{more than one driver arrives during } [t, t+\Delta]).
\end{align}
By conditioning on the location of the arrival we further get
\begin{align}
&N(t+\Delta,n)1(\text{bundle size $w$ arrives during }[t, t+\Delta]) \notag \\
&=\sum_{i=1}^n N(t+\Delta,n)1(\text{bundle of size $w$ arrives at position $i$ during }[t, t+\Delta]) \notag \\
&= 1(n\geq w) \sum_{i=1}^{n-w+1} N(t+\Delta ,n)1(\text{bundle of size $w$ arrives at position $i$ during }[t, t+\Delta]) \notag \\
&\hspace{5mm}+ 1(n\geq w) \sum_{i= n-w+2 }^{n} N(t ,n)1(\text{bundle of size $w$ arrives at position $i$ during }[t, t+\Delta]), \label{eq:RejectArrival1} \\
&\hspace{5mm}+ 1(n< w) \sum_{i= 1}^{n} N(t ,n)1(\text{bundle of size $w$ arrives at position $i$ during }[t, t+\Delta]), \label{eq:RejectArrival2}
\end{align}
where \eqref{eq:RejectArrival1} and \eqref{eq:RejectArrival2} correspond to the cases where we reject the arrival since either some packages in the requested bundle are not available or the bundle size is bigger than the number of remaining packages.  Taking expectation on both sides of \eqref{eq1}, we have
\begin{align*}
&\mathcal{K}(t+\Delta,n) \\
&= E\left[ N(t,n) 1(\text{no drivers arrive during }[t, t+\Delta]) \right] + O( \Delta^2 n^3) \\
&\hspace{5mm} + \sum_{w=1}^m 1(n \geq w) \sum_{i=1}^{n-w+1}  E\left[ N(t+\Delta ,n)1(\text{bundle of size $w$ arrives at position $i$ during }[t, t+\Delta]) \right] \\
&\hspace{5mm} + \sum_{w=1}^m 1(n \geq w) \sum_{i=n-w+2}^n E[ N(t,n) 1(\text{bundle of size $w$ arrives at position $i$ during }[t, t+\Delta]) ] \\
&\hspace{5mm} + \sum_{w=1}^m 1(n < w) \sum_{i=1}^n E[ N(t,n) 1(\text{bundle of size $w$ arrives at position $i$ during }[t, t+\Delta]) ] \\
&= \mathcal{K}(t,n) P(\text{no drivers arrive during }[t, t+\Delta]) + O( \Delta^2 n^3)  \\
&\hspace{5mm} +  \sum_{w=1}^m 1(n \geq w) \sum_{i=1}^{n-w+1}  \{ \mathcal{K}(t,i-1) + \mathcal{K}(t,n-i-w+1) + w \} \\
&\hspace{35mm} \cdot P(\text{bundle of size $w$ arrives at position $i$ during }[t, t+\Delta])  \\
&\hspace{5mm} + \sum_{w=1}^m 1(n \geq w) \sum_{i=n-w+2}^n \mathcal{K}(t,n) P(\text{bundle of size $w$ arrives at position $i$ during }[t, t+\Delta])  \\
&\hspace{5mm} + \sum_{w=1}^m 1(n <  w) \sum_{i=1}^n \mathcal{K}(t,n) P(\text{bundle of size $w$ arrives at position $i$ during }[t, t+\Delta]) \\
&= \mathcal{K}(t,n) (1 - \lambda n\Delta) + O( \Delta^2 n^3)  \\
&\hspace{5mm} + \sum_{w=1}^m 1(n \geq w) \sum_{i=1}^{n-w+1} \{ \mathcal{K}(t,i-1) + \mathcal{K}(t,n-i-w+1) + w \} f(w) \cdot \frac{1}{n} \cdot (\lambda n  \Delta)  \\
&\hspace{5mm} + \sum_{w=1}^m 1(n \geq w) \sum_{i=n-w+2}^n \mathcal{K}(t,n) f(w)\cdot \frac{1}{n}   \cdot (\lambda n  \Delta) \\
&\hspace{5mm} + \sum_{w=1}^m 1(n < w) \sum_{i=1}^n \mathcal{K}(t,n) f(w)\cdot \frac{1}{n}   \cdot (\lambda n  \Delta) ,
\end{align*}
where in the second equality we used the observation that 
\begin{align*}
&E[ N(t+\Delta, n) | \text{bundle size $w$ arrives at position $i$ during }[t, t+\Delta]] \\
&= \mathcal{K}(t,i-1) + \mathcal{K}(t,n-i-w+1) + w
\end{align*}
when bundle at location $i$ is available for pick-up. The term $O(\Delta^2 n^3)$ is relative to the limit $\Delta \to 0$ and it includes the probability of having two or more arrivals during $[t, t+\Delta]$. 

We have thus shown that
\begin{align*}
\mathcal{K}(t+\Delta,n) - \mathcal{K}(t,n) &= - \lambda n \Delta \mathcal{K}(t,n) + O( \Delta^2 n^3)  \\
&\hspace{5mm} +  \sum_{w=1}^m 1(n \geq w) \sum_{i=1}^{n-w+1} \{ \mathcal{K}(t,i-1) + \mathcal{K}(t,n-i-w+1) + w \} f(w)  \lambda   \Delta  \\
&\hspace{5mm} + \sum_{w=1}^m \sum_{i=(n-w+2) \vee 1}^n \mathcal{K}(t,n) f(w)  \lambda  \Delta \\
&=  - \lambda n \Delta \mathcal{K}(t,n) +  \sum_{w=1}^m 1(n \geq w) \left\{ 2\sum_{i=0}^{n-w} \mathcal{K}(t,i) + w(n-w+1) \right\} f(w) \lambda \Delta \\
&\hspace{5mm} + \sum_{w=1}^m (n \wedge (w-1)) \mathcal{K}(t,n) f(w)  \lambda  \Delta ,
\end{align*}
which yields the differential equation:
\begin{align*}
\frac{\partial}{\partial t} \mathcal{K}(t,n) &= - \lambda n  \mathcal{K}(t,n) + \sum_{w=1}^m 1(n \geq w) \left\{ 2\sum_{i=0}^{n-w} \mathcal{K}(t,i) + w(n-w+1) \right\} f(w) \lambda \\
&\hspace{5mm} + \sum_{w=1}^m (n \wedge (w -1)) \mathcal{K}(t,n) f(w)  \lambda \\
&= - \lambda n  \mathcal{K}(t,n) +  2\lambda \sum_{w=1}^{m \wedge n} \sum_{i=0}^{n-w} \mathcal{K}(t,i) f(w)  + \lambda \sum_{w=1}^{m\wedge n}  w(n-w+1) f(w)  \\
&\hspace{5mm} +  \sum_{w=1}^m (n \wedge (w -1)) \mathcal{K}(t,n) f(w)  \lambda  \\
&= -\lambda n \mathcal{K}(t,n) + 2\lambda \sum_{w=1}^n \sum_{i=0}^{n-w} \mathcal{K}(t,i) f(w)   \\
&\hspace{5mm} + \lambda \sum_{w=1}^{n}  w(n-w+1) f(w)  + \lambda \mathcal{K}(t,n) \sum_{w=1}^m (n \wedge (w -1))  f(w)  ,
\end{align*}
and in the third equality we used the observation that since $f(w) = 0$ for $w > 0$, we can replace the upper limit in the sum by $n$. 

To further simplify the expression, exchange the order of the sums to obtain that
$$2\lambda \sum_{w=1}^{n} \sum_{i=0}^{n-w} \mathcal{K}(t,i) f(w)  = 2\lambda \sum_{i=0}^{n-1} \mathcal{K}(t,i) \sum_{w=1}^{n-i}  f(w) = 2\lambda \sum_{i=0}^{n-1} \mathcal{K}(t,i)  F(n-i). $$
Also, 
\begin{align*}
&- n + \sum_{w=1}^m (n \wedge (w -1))  f(w) \\
&=  - n  + \sum_{w=1}^{n \wedge m} (w-1) f(w) + 1(m > n) \sum_{w= n + 1}^m n f(w) \\
&= - n + \sum_{w=1}^{n \wedge m} (w-1) (F(w) - F(w-1)) + 1(m > n) n (1 - F(n)) \\
&= - n +  \sum_{w=1}^{n \wedge m} w F(w) - \sum_{w=0}^{(n\wedge m) - 1} w F(w) - \sum_{w=1}^{n \wedge m} F(w) + 1(m > n) n (1 - F(n)) \\
&= - n + (n \wedge m) F( n \wedge m) - \sum_{w=1}^{n \wedge m} F(w) + 1(m > n) n (1 - F(n)) \\
&= \begin{cases}  - \sum_{w=1}^n F(w) , & n < m , \\
-n + m - \sum_{w=1}^m F(w), & n \geq m, \end{cases} \\
&= - \sum_{w=1}^n F(w). 
\end{align*}

We conclude that
\begin{equation*}
\frac{1}{\lambda}\frac{\partial\mathcal{K}(t,n)}{\partial t}=-\sum_{w=1}^n F(w) \mathcal{K}(t,n)+2\sum_{i=1}^{n-1}F(n-i)\mathcal{K}(t,i)+\sum_{w=1}^n w(n-w+1) f(w)
\end{equation*}
with boundary condition $\mathcal{K}(0,n)=0$.

Finally, let $\mathcal{R}(t,n) = n - \mathcal{K}(t,n)$ denote the expected number of undelivered packages at time $t$. Then, the above ODE becomes:
\begin{align*}
\frac{1}{\lambda}\frac{\partial\mathcal{R}(t,n)}{\partial t} &= -\sum_{i=1}^n F(i) \mathcal{R}(t,n)  + 2\sum_{i=1}^{n-1}F(n-i) \mathcal{R}(t,i) + n\sum_{i=1}^n F(i)  \\
&\hspace{5mm} - 2\sum_{i=1}^{n-1}F(n-i) i - \sum_{i=1}^n i(n-i+1) f(i) \\
&= -\sum_{i=1}^n F(i) \mathcal{R}(t,n)  + 2\sum_{i=1}^{n-1}F(n-i) \mathcal{R}(t,i) + n\sum_{j=1}^n f(j) (n-j+1)   \\
&\hspace{5mm} - \sum_{j=1}^{n-1} f(j)  (n-j)(n-j+1)  - \sum_{i=1}^n i(n-i+1) f(i) \\
&= -\sum_{i=1}^n F(i) \mathcal{R}(t,n)  + 2\sum_{i=1}^{n-1}F(n-i) \mathcal{R}(t,i) .   
\end{align*}
The corresponding boundary condition is $\mathcal{R}(0,n) = n$. This completes the proof. 
\end{proof}

\bigskip

After obtaining the differential equation, we use induction to prove Proposition~\ref{recursive}, which provides the explicit solution to the ODE from Theorem~\ref{ode}.

\bigskip

\begin{proof}[Proof of Proposition~\ref{recursive}]
Note that for any $n$ such that $F(n)=0$, $\mathcal{R}(t,n)=n$. Thus the boundary condition for $\gamma_{n,j}=1$ satisfies
$$\mathcal{R}(t,n) = \sum_{j=1}^n \gamma_{n,j}e^{-\lambda\sum_{k=1}^j F(k)t}=\sum_{j=1}^n \gamma_{n,j} = n.$$

It remains to prove the result for $n$ such that $F(n) > 0$. To start, let
$$\phi_n(t)=2\lambda\sum_{i=1}^{n-1}F(n-i)\mathcal{R}(t,i)$$
$$\theta_n=\lambda\sum_{j=1}^n \overline{F}(j) $$ 
then the above differential equation becomes
$$\frac{d\mathcal{R}(t,n)}{dt}+ (\lambda n-\theta_n) \mathcal{R}(t,n)=\phi_n(t)$$
with boundary condition $\mathcal{R}(0,n) = n$. We will prove by induction in $n$ that 
\begin{equation} \label{eq:InductionHyp}
\mathcal{R}(t,n)= \sum_{i=1}^n \gamma_{n,i} e^{-(\lambda i - \theta_i) t},
\end{equation}
where 
$$\gamma_{n,n}= n -\sum_{i=1}^{n-1}  \gamma_{n,i}  \quad\text{and}\quad\gamma_{n,i}= 2 \cdot \frac{\sum_{j=1}^{n-i}F(j) \gamma_{n-j,i}}{\sum_{k=i+1}^{n}F(k)}, \quad 1 \leq i < n, \qquad \gamma_{1,1} = 1.$$

Suppose now that \eqref{eq:InductionHyp} holds and consider $\mathcal{K}(t,n+1)$. Solving the differential equation satisfied by $\mathcal{R}(t,n+1)$ directly (see, e.g., \cite{schaeffer2016ordinary}), we obtain that
\begin{align}
&\mathcal{R}(t,n+1) \notag \\
&= e^{-(\lambda(n+1) - \theta_{n+1}) t} \left( n+1 + \int_0^t \phi_{n+1}(v) e^{ v (\lambda(n+1) - \theta_{n+1})} dv \right)  \notag \\
&= e^{-(\lambda (n+1) - \theta_{n+1}) t} (n+1) + \lambda e^{-(\lambda (n+1) - \theta_{n+1}) t} \int_{0}^t 2\sum_{i=1}^n F(n+1-i)\mathcal{R}(v,i) e^{(\lambda(n+1) - \theta_{n+1}) v}dv \notag \\
&= e^{-(\lambda (n+1) - \theta_{n+1}) t} (n+1) + \lambda e^{-(\lambda (n+1) - \theta_{n+1}) t} \int_{0}^t  2\sum_{i=1}^n F(n+1-i) \sum_{j=1}^i \gamma_{i,j} e^{-(\lambda j - \theta_j)  v}  e^{ (\lambda (n+1) - \theta_{n+1}) v} dv \notag \\
&=  e^{-(\lambda (n+1) - \theta_{n+1}) t} (n+1) + e^{-(\lambda (n+1) - \theta_{n+1}) t}  2\lambda \sum_{i=1}^n F(n+1-i) \sum_{j=1}^i \gamma_{i,j}  \cdot  \frac{e^{ (\lambda (n+1 -j) - \theta_{n+1} + \theta_j) t} -1 }{\lambda(n+1-j) - \theta_{n+1} + \theta_j}  \notag \\
&=  e^{-(\lambda (n+1) - \theta_{n+1}) t} (n+1) +  2 \sum_{i=1}^n F(n+1-i) \sum_{j=1}^i \gamma_{i,j}  \cdot  \frac{e^{ -(\lambda j - \theta_j) t} - e^{-(\lambda (n+1) - \theta_{n+1}) t} }{\sum_{k=j+1}^{n+1} F(k) }  , \notag
\end{align}
where in the third equality we used the induction hypothesis \eqref{eq:InductionHyp}.

It remains to analyze the last expression, for which we exchange the summation order to obtain that
\begin{align*}
\mathcal{R}(t,n+1) &= e^{-(\lambda (n+1) - \theta_{n+1}) t} (n+1) + 2 \sum_{j=1}^n  \sum_{i=j}^n F(n+1-i) \gamma_{i,j}  \cdot  \frac{e^{ -(\lambda j - \theta_j) t} - e^{-(\lambda (n+1) - \theta_{n+1}) t} }{\sum_{k=j+1}^{n+1} F(k) } \\
&= e^{-(\lambda (n+1) - \theta_{n+1}) t} (n+1) + 2 \sum_{j=1}^n \frac{e^{ -(\lambda j - \theta_j) t} - e^{-(\lambda (n+1) - \theta_{n+1}) t}}{\sum_{k=j+1}^{n+1} F(k) }  \sum_{r=1}^{n+1-j} F(r) \gamma_{n+1-r,j}  \\
&=  e^{-(\lambda (n+1) - \theta_{n+1}) t} (n+1) + \sum_{j=1}^n \left( e^{ -(\lambda j - \theta_j) t} - e^{-(\lambda (n+1) - \theta_{n+1}) t} \right) \gamma_{n+1,j} \\
&= e^{-(\lambda (n+1) - \theta_{n+1}) t} (n+1) +  \sum_{i=1}^n \gamma_{n+1,i} e^{-(\lambda i - \theta_i)t} - e^{-(\lambda(n+1) - \theta_{n+1})t} \sum_{i=1}^n \gamma_{n+1,i}.
\end{align*}
To complete the proof for $\mathcal{R}(t,n)$ note that
$$\sum_{i=1}^n \gamma_{n+1,i} = n+1 - \gamma_{n+1,n+1}$$
gives
\begin{align*}
\mathcal{R}(t,n+1) &= e^{-(\lambda (n+1) - \theta_{n+1}) t} (n+1) +  \sum_{i=1}^n \gamma_{n+1,i} e^{-(\lambda i - \theta_i)t} - e^{-(\lambda(n+1) - \theta_{n+1})t} (n+1) \\
&\hspace{5mm} + e^{-(\lambda(n+1) - \theta_{n+1})t} \gamma_{n+1,n+1} \\
&=  \sum_{i=1}^{n+1} \gamma_{n+1,i} e^{-(\lambda i - \theta_i)t} . 
\end{align*}

We now use the explicit expression for $\mathcal{R}(t,n)$ to compute $\mathcal{C}(t,n)$. Recall that $T_n^*$ is denotes the time of the first request when we start with $n$ packages arranged on a circle, and $B$ is the size of the corresponding bundle. Moreover, since $T_n^*$ is exponentially distributed with rate $\lambda n$ and $B$ has distribution $F$, we have 
\begin{align*}
\mathcal{C}(t,n)&=E[(B+\mathcal{K}(t-T_n^*,n-B))1(T_n^*<t)] \\
&= E\left[ \left( n - \mathcal{R}(t-T_n^*, n-B) \right) 1(T_n^* < t) \right] \\
&=nP(T_n^* <t)-\sum_{k=1}^m f(k)\sum_{i=1}^{n-k}\gamma_{n-k,i}e^{-\lambda\sum_{j=1}^i F(j)t}E\left[e^{\lambda \sum_{j=1}^i F(j)T_n^*} 1(T_n^*<t) \right] \\
&=n(1-e^{-\lambda nt})-\sum_{k=1}^m\sum_{i=1}^{n-k}f(k)\gamma_{n-k,i}\frac{n}{n-\sum_{j=1}^i F(j)}\left(1-e^{\lambda(\sum_{j=1}^i F(j)-n)t}\right)e^{-\lambda\sum_{j=1}^i F(j)t}\\
&=n(1-e^{-\lambda nt})- \sum_{i=1}^{n-1}\sum_{k=1}^{n-i}f(k)\gamma_{n-k,i}\frac{n}{n-\sum_{j=1}^i F(j)}\left(e^{-\lambda\sum_{j=1}^i F(j) t}-e^{-\lambda nt}\right)\\
&=n- \sum_{i=1}^{n-1} \tilde\gamma_{n,i}e^{-\lambda\sum_{j=1}^i F(j) t}-\tilde \gamma_{n,n}e^{-\lambda nt},
\end{align*}
where 
$$\tilde \gamma_{n,i}=\sum_{k=1}^{n-i}f(k)\frac{\gamma_{n-k,i}}{1-\frac{1}{n}\sum_{j=1}^i F(j)}\quad\text{ and }\quad\tilde \gamma_{n,n}=n-\sum_{i=1}^{n-1}\tilde \gamma_{n,i}.$$
\end{proof}

\bigskip

The following technical result provides a monotonicity property for $\mathcal{R}(t,n)$ that will be needed for the application of a Tauberian theorem in the proof of Theorem~\ref{alpha}.  Interestingly, $\mathcal{R}(t,n)$ is not generally monotone in $n$.

\begin{definition}
We say that a sequence $\{a_n: n \geq 1\}$ is eventually increasing (decreasing) if there exists an $n_0 \in \mathbb{N}_+$ such that $a_n$ is increasing (decreasing) for all $n \geq n_0$. 
\end{definition}

\begin{lemma}\label{MonotonicDiff}
For any $t \geq 0$, 
 $$\mathcal{K}(t,n)+ n \cdot \frac{m+1}{m-\mu+1}$$
is monotonically increasing in $n$ when $n\geq m$, where $\mu = E[B]$ and $m$ is the maximum bundle size. 
\end{lemma}

\begin{proof}
Recall that $\mathcal{K}(t,n)$ is the expected number of packages that can be picked up by time $t$ when we start with $n$ packages arranged on a line. For $n\geq m$, note that arriving requests at location $i$ occur according to a Poisson process with rate $\lambda$ when $i\leq n-m+1$ and with rate $\lambda F(n+1-i)$ when $i>n-m+1$, the latter since close to the right end-point bundle sizes need to be smaller than $n-i$ to be accepted. Furthermore, the time $T_n$ at which the first request is accepted is independent of the location where it occurs, and is exponentially distributed with rate $\lambda (n-m+1) + \lambda \sum_{i=1}^{m-1} F(i) = \lambda (n-\mu+1)$. Thus,  the probability that the first accepted request occurs at location $i$, denoted as $p_i^{(n)}$, is 
$$p_i^{(n)} = \frac{1}{n-\mu+1}, \qquad \text{for }  1 \leq i\leq n-m+1,$$
and 
$$p_i^{(n)}=\frac{F(n+1-i)}{n-\mu+1}, \qquad \text{for } n-m+1 < i \leq n.$$
Let $\mathcal{Q}(t,k)=E[B+N(t,k-B)|B\leq k]$, where $B$ is the bundle size (distributed according to $F$) and $N(t,k)$ is the number of packages picked up during $[0,t]$ when $k$ packages are arranged on a line, and note that $\mathcal{Q}(t,k)$ denotes the expected number of packages that will be picked-up during the interval $[0, t]$ given that a request has been accepted at time zero at location 1 of a total of $k$ packages. As in previous proofs, we have dropped the $\lambda$ from the notation. Let $L_n$ denote the location of the first request to be accepted when we start with $n$ packages. 
To analyze $\mathcal{K}(t,n)$, we condition on $T_n$ and $L_n$ to obtain
\begin{align*}
\mathcal{K}(t,n)&=E[E[N(t,n)|T_n]]=E\left[\sum_{i=1}^n P(L_n =i|T_n)E[N(t,n)|T_n,L_n =i]\right]\\
&= E\left[\sum_{i=1}^n p_i^{(n)}\left(E[N((t-T_n)^+,i-1) | T_n] \right. \right. \\
&\hspace{5mm} \left.  \left. +E\left[B+N((t-T_n)^+,n-i+1-B)|B\leq n-i+1, T_n \right]\right)\right]\\
&=E\left[\sum_{i=1}^{n} p^{(n)}_i \left(\mathcal{K}((t-T_n)^+,i-1)+\mathcal{Q}((t-T_n)^+,n-i+1)\right)\right]\\
&=\frac{1}{n-\mu+1}E\left[\sum_{i=1}^{n-m+1}  \left( \mathcal{K}((t-T_n) ^+, i-1)+ \mathcal{Q}((t-T_n)^+,n-i+1) \right) \right]\\
&\hspace{5mm}+\sum_{j=1}^{m-1}\frac{F(m-j)}{n-\mu+1}E\left[\mathcal{K}((t-T_n)^+,n-m+j)+\mathcal{Q}((t-T_n)^+,m-j)\right].
\end{align*}
Therefore, we get the difference
\begin{align*}
&(n-\mu+2)\mathcal{K}(t,n+1)-(n-\mu+1)\mathcal{K}(t,n)\\
&=E\left[\sum_{i=1}^{n-m+1} \mathcal{K}((t-T_{n+1})^+, i-1)+  \sum_{i=2}^{n-m+2} \mathcal{Q}((t-T_{n+1})^+,n-i+2)\right]\\
&\hspace{5mm} + E[\mathcal{K}((t-T_{n+1})^+,n-m+1)+\mathcal{Q}((t-T_{n+1})^+,n+1)] \\
&\hspace{5mm}+\sum_{j=1}^{m-1}F(m-j) E\left[\mathcal{K}((t-T_{n+1})^+,n-m+j+1)+\mathcal{Q}((t-T_{n+1})^+,m-j)\right]\\
&\hspace{5mm}-E\left[\sum_{i=1}^{n-m+1} \left( \mathcal{K}((t-T_n) ^+, i-1)+ \mathcal{Q}((t-T_n)^+,n-i+1) \right) \right]\\
&\hspace{5mm}-\sum_{j=1}^{m-1}F(m-j) E\left[\mathcal{K}((t-T_n)^+,n-m+j)+\mathcal{Q}((t-T_n)^+,m-j)\right]  \\
&= \sum_{i=1}^{n-m+1} \left( D(t,n,i-1) + \hat D(t,n,n-i+1) \right) + \sum_{k=1}^{m-1} F(k) \hat D(t,n,k)  \\
&\hspace{5mm} + E[\mathcal{K}((t-T_{n+1})^+,n-m+1)+\mathcal{Q}((t-T_{n+1})^+,n+1)]  \\
&\hspace{5mm}+ \sum_{k=1}^{m-1} F(k) E\left[ \mathcal{K}((t-T_{n+1})^+,n-k+1)\right]  - \sum_{k=1}^{m-1} F(k) E\left[ \mathcal{K}((t-T_{n})^+,n-k) \right] \\
&= \sum_{i=1}^{n-m+1} \left( D(t,n,i-1) + \hat D(t,n,n-i+1) \right) + \sum_{k=1}^{m-1} F(k) \left( \hat D(t,n,k)  + D(t,n,n-k) \right) \\
&\hspace{5mm} + E[\mathcal{K}((t-T_{n+1})^+,n-m+1)+\mathcal{Q}((t-T_{n+1})^+,n+1)]  \\
&\hspace{5mm} + \sum_{k=0}^{m-1} F(k) E\left[ \mathcal{K}((t-T_{n+1})^+,n-k+1) - \mathcal{K}((t-T_{n+1})^+, n-k) \right] ,
\end{align*}
where 
\begin{align*}
D(t,n,i) &:= E\left[ \mathcal{K}((t-T_{n+1})^+,i) - \mathcal{K}((t-T_{n})^+,i) \right] \\
\hat D(t,n,i) &:= E\left[ \mathcal{Q}((t-T_{n+1})^+,i) - \mathcal{Q}((t-T_{n})^+,i) \right].
\end{align*}
Furthermore, since $F(0) = 0$ and $F(m) = 1$ we have that
\begin{align*}
&E[\mathcal{K}((t-T_{n+1})^+,n-m+1) ] +  \sum_{k=0}^{m-1} F(k) E\left[ \mathcal{K}((t-T_{n+1})^+,n-k+1) - \mathcal{K}((t-T_{n+1})^+, n-k) \right]  \\
&= \sum_{k=0}^{m} F(k) E\left[ \mathcal{K}((t-T_{n+1})^+,n-k+1) - \mathcal{K}((t-T_{n+1})^+, n-k) \right] + E[\mathcal{K}((t-T_{n+1})^+,n-m) ] \\
&= \sum_{k=1}^m F(k) E\left[ \mathcal{K}((t-T_{n+1})^+,n-k+1) \right] - \sum_{k=1}^{m+1} F(k-1) E\left[ \mathcal{K}(t-T_{n+1})^+,n-k+1) \right] \\
&\hspace{5mm} + E[\mathcal{K}((t-T_{n+1})^+,n-m) ] \\
&= \sum_{k=1}^m f(k) E\left[ \mathcal{K}((t-T_{n+1})^+,n-k+1) \right]  .
\end{align*}

Now use the observation that for any $j \geq m$ and $t \geq 0$ we have
\begin{align*}
 \mathcal{Q}(t,j) &= E\left[ B + N(t,j- B) | B \leq j \right] \\
&= E[B] + E[ \mathcal{K}(t,j-B)] = \mu + \sum_{k=1}^m f(k) \mathcal{K}(t, j-k),
\end{align*}
to obtain that
\begin{align*}
E[\mathcal{Q}((t-T_{n+1})^+,n+1)] =  \mu + \sum_{k=1}^m f(k)  \mathcal{K}((t-T_{n+1})^+,n-k+1)] 
\end{align*}
and
$$\hat D(t,n,j) = \sum_{k=1}^m f(k) D(t,n,j-k) \qquad \text{for } j \geq m.$$
We have thus derived that for $n \geq m$,
\begin{align} \label{eq:MainIdentity}
&(n-\mu+2)\mathcal{K}(t,n+1)-(n-\mu+1)\mathcal{K}(t,n) \notag \\
&=  \sum_{i=1}^{n-m+1} \left( D(t,n,i-1) + \sum_{k=1}^m f(k)  D(t,n,n-i+1-k) \right) + \sum_{i=1}^{m-1} F(i) \left( D(t,n,n-i) + \hat D(t,n,i)  \right) \notag \\
&\hspace{5mm} + 2 \sum_{i=1}^m f(i) E\left[ \mathcal{K}((t-T_{n+1})^+,n-i+1) \right] + \mu. 
\end{align}

Finally, note that $T_{n+1}\leq_{s.t.}T_n$, where (s.t.) denotes the standard stochastic order, which since  both $\mathcal{K}$ and $\mathcal{Q}$ are non-decreasing in $t$, implies that $D(t,n,i) \geq 0$ and $\hat D(t,n,j) \geq 0$ for any $i,j \geq 1$. Hence, we immediately obtain that
$$(n-\mu+2)\mathcal{K}(t,n+1)-(n-\mu+1)\mathcal{K}(t,n) \geq 0.$$
Dividing by $n-\mu+1$ now gives
$$\mathcal{K}(t,n+1) - \mathcal{K}(t,n) \geq - \frac{1}{n-\mu+1} \cdot \mathcal{K}(t,n+1),$$
and using the observation that $\mathcal{K}(t,n+1) \leq n +1$ further gives
$$\mathcal{K}(t,n+1) - \mathcal{K}(t,n) \geq - \frac{n+1}{n-\mu+1} \geq - \frac{m+1}{m-\mu+1}$$
for all $n \geq m$. This in turn implies that 
$$\mathcal{K}(t,n)+n \cdot \frac{m+1}{m-\mu+1}$$
is monotonically increasing with $n\geq m$. 
\end{proof}

\bigskip

Next we calculate the formula for $\beta(t) = 1 - \alpha(t)$ through the generating function $G(t,x):=\sum_{n=m}^\infty \mathcal{R}(t,n)x^n$. The key tool in the analysis is the use of a Tauberian theorem that allows us to infer the behavior of $\mathcal{R}(t,n)$ from that of $G(t,x)$. We write $f(x) \sim g(x)$ as $x \to a$ to denote $\lim_{x \to a} f(x)/g(x) = 1$.

\bigskip

\begin{proof}[Proof of Theorem~\ref{alpha}]
Define $\varphi(y) = \sum_{i=1}^{m-1} \overline{F}(i) y^i/i$ and recall that $\theta_n=\lambda\sum_{j=1}^n \overline{F}(j) $. Note that when $n \geq m$ we have $\theta_n = \lambda \varphi'(1)$. From Theorem~\ref{ode} we have
\begin{align*}
\frac{\partial\mathcal{R}(t,n)}{\partial t} &= - \lambda \sum_{i=1}^n F(i) \mathcal{R}(t,n)  + 2 \lambda \sum_{i=1}^{n-1}F(n-i) \mathcal{R}(t,i) \\
&= - (\lambda n - \theta_n) \mathcal{R}(t,n) +  2 \lambda \sum_{i=1}^{n-1}F(n-i) \mathcal{R}(t,i) . 
\end{align*}
Next, multiply both sides by $x^n$ and sum over $n$ from $m$ to infinity to obtain
\begin{align*}
\sum_{n=m}^\infty \frac{\partial\mathcal{R}(t,n)}{\partial t} x^n &= - \sum_{n=m}^\infty (\lambda n - \lambda \varphi'(1)) \mathcal{R}(t,n) x^n + 2\lambda \sum_{n=m}^\infty \sum_{i=1}^{n-1} F(n-i) \mathcal{R}(t,i) x^n \\
&= -\lambda x \sum_{n=m}^\infty \mathcal{R}(t,n) \frac{d}{dx} x^n + \lambda \varphi'(1) G(t,x) + 2\lambda \sum_{i=1}^{m-1} \sum_{n=m}^\infty  F(n-i) \mathcal{R}(t,i) x^n \\
&\hspace{5mm} + 2\lambda \sum_{i=m}^\infty \sum_{n=i+1}^\infty F(n-i) \mathcal{R}(t,i) x^n \\
&= -\lambda x \frac{\partial}{\partial x} G(t,x) + \lambda \varphi'(1) G(t,x) + 2\lambda \sum_{i=1}^{m-1} \mathcal{R}(t,i) x^i \sum_{j=m-i}^\infty  F(j)  x^{j} \\
&\hspace{5mm} + 2\lambda \sum_{i=m}^\infty \mathcal{R}(t,i) x^i \sum_{j=1}^\infty F(j)  x^{j} \\
&= -\lambda x \frac{\partial}{\partial x} G(t,x) + \lambda \varphi'(1) G(t,x) + 2 \lambda \sum_{i=1}^{m-1} \mathcal{R}(t,i) x^i \left( \sum_{j=m-i}^\infty x^j -  \sum_{j=m-i}^\infty  \overline{F}(j)  x^{j} \right) \\
&\hspace{5mm} + 2 \lambda G(t,x) \left(  \sum_{j=1}^\infty x^j - x \varphi'(x) \right) \\
&=  -\lambda x \frac{\partial}{\partial x} G(t,x) + \lambda \left( \varphi'(1) + 2 x (1-x)^{-1} - 2x \varphi'(x) \right) G(t,x) \\
&\hspace{5mm} + 2 \lambda x^{m} (1-x)^{-1} \sum_{i=1}^{m-1} \mathcal{R}(t,i)  - 2 \lambda \sum_{i=1}^{m-1} \mathcal{R}(t,i) x^i  \sum_{j=m-i}^\infty \overline{F}(j) x^j,
\end{align*}
where the exchange of derivative and series in the third equality is justified by Theorem A.5.1 in \cite{durrett2010probability} and we use the convention that $\sum_{i=a}^b x_i \equiv 0$ if $b < a$. Furthermore, Theorem~A.5.1 in \cite{durrett2010probability} also gives that
$$\sum_{n=m}^{\infty} \frac{\partial \mathcal{R}(t, n)}{\partial t}x^n = \frac{\partial G(t,x)}{\partial t},$$
and we obtain that $G(t,x)$ satisfies the differential equation:
\begin{align*}
\frac{\partial G(t,x)}{\partial t}  &=  -\lambda x \frac{\partial}{\partial x} G(t,x) + \lambda \left( \varphi'(1) + 2 x (1-x)^{-1} - 2x \varphi'(x) \right) G(t,x) \\
&\hspace{5mm} + 2 \lambda x^{m} (1-x)^{-1} \sum_{i=1}^{m-1} \mathcal{R}(t,i)  - 2 \lambda \sum_{i=1}^{m-1} \mathcal{R}(t,i) x^i  \sum_{j=m-i}^\infty \overline{F}(j) x^j. 
\end{align*}

To solve it, make the change of variables $r = \ln x - \lambda t$ and $s = t$, and define $\tilde G(s,r) = G( s, e^{\lambda s + r})$ to obtain
\begin{align*}
\frac{\partial}{\partial t} G(t,x) &= \frac{\partial}{\partial r} \tilde G(s,r) (-\lambda) + \frac{\partial}{\partial s} \tilde G(s,r) , \quad \text{and} \\
\frac{\partial}{\partial x} G(t,x) &= \frac{\partial}{\partial r} \tilde G(s,r) \frac{1}{x} .
\end{align*}
Substituting in our expression for $\frac{\partial}{\partial t} G(t,x)$ we obtain
\begin{align*}
& \frac{\partial}{\partial r} \tilde G(s,r) (-\lambda) + \frac{\partial}{\partial s} \tilde G(s,r) \\
&= - \lambda \frac{\partial}{\partial r} \tilde G(s,r) +   \tilde G(s,r) \lambda \left( \varphi'(1) + \frac{2 e^{\lambda s+ r}}{1-e^{\lambda s + r}} - 2 e^{\lambda s+r} \varphi'(e^{\lambda s+r})      \right)   \\
&\hspace{5mm}  + \frac{2 \lambda e^{m(\lambda s+ r)}}{1-e^{\lambda s + r}} \sum_{i=1}^{m-1} \mathcal{R}(s,i) - 2 \lambda \sum_{i=1}^{m-1} \mathcal{R}(s,i) \sum_{j=m-i}^\infty \overline{F}(j) e^{(i+j)(\lambda s+r)}   ,
\end{align*}
which by cancelling the terms $\lambda \frac{\partial}{\partial r} \tilde G(s,r)$ on both sides can be written as 
$$\frac{d}{ds} H(s) + H(s) p(s) = q(s),$$
with $H(s) = \tilde G(s,r)$, 
$$p(s) = - \lambda \left( \varphi'(1) + \frac{2 e^{\lambda s+ r}}{1-e^{\lambda s + r}} - 2 e^{\lambda s+r} \varphi'(e^{\lambda s+r})      \right),$$
and
\begin{align*}
q(s) &= \lambda \left( \frac{2 e^{m(\lambda s+ r)}}{1-e^{\lambda s + r}} \sum_{i=1}^{m-1} \mathcal{R}(s,i) - 2\sum_{i=1}^{m-1} \mathcal{R}(s,i) \sum_{j=m-i}^{m-1} \overline{F}(j) e^{(i+j)(\lambda s+r)}  \right).
\end{align*}

The solution of this ODE (see \cite{schaeffer2016ordinary} P.8 equation (1.18)) is 
$$H(s)=e^{-\int_0^sp(h)dh}\left(C+\int_{0}^sq(v)e^{\int_{0}^vp(h)dh}dv\right),$$
where $C$ is a constant determined by the boundary conditions. Since $H(0)=\tilde G(0,r) = G(0, e^r) = \sum_{n=m}^\infty n e^{rn} $, we have
$$H(s)=e^{-\int_0^sp(h)dh}\left( H(0) + \int_{0}^sq(v)e^{\int_{0}^vp(h)dh}dv \right).$$

Substituting the expression for $p(s)$ gives
\begin{align*}
-\int_0^s p(h)dh &= \lambda \int_0^s \left( \varphi'(1) + \frac{2 e^{\lambda h+ r}}{1-e^{\lambda h + r}} - 2 e^{\lambda h+r} \varphi'(e^{\lambda h+r})      \right) dh \\
&=\lambda \varphi'(1) s+2\int_{e^r}^{e^{\lambda s+r}} \left(  \frac{1}{1-y}- \varphi'(y)  \right) dy\\
&=\lambda \varphi'(1) s + 2\left(\ln(1 - e^r) - \ln(1 - e^{r+\lambda s}) \right)-2\left(\varphi(e^{r+\lambda s})-\varphi(e^r)\right).
\end{align*}
This in turn implies that if we define $h(s,x,y) = (1-y)^2 e^{-2(\varphi(x) - \varphi(y)) + \lambda \varphi'(1) s}$, then
$$e^{-\int_0^sp(h)dh}=\frac{(1-e^{r})^2}{(1-e^{r+\lambda s})^2} e^{-2\left(\varphi(e^{r+\lambda s})-\varphi(e^r)\right)+\lambda \varphi'(1) s} = \frac{1}{(1-e^{r+\lambda s})^2 } \cdot h(s, e^{r+\lambda s}, e^r) .$$
Define also
\begin{align*}
\tilde q(x,y) &:=  2 y^m (1-y) \sum_{i=1}^{m-1} \mathcal{R}(\ln(y/x)/\lambda,i) - 2 (1-y)^2 \sum_{i=1}^{m-1} \mathcal{R}(\ln( y/x)/\lambda,i) \sum_{j=m-i}^{m-1} \overline{F}(j) y^{i+j} .
\end{align*}
and note that $q(v) = \tilde q(e^r, e^{r+\lambda v})/(1-e^{r+\lambda v})^2$.  It follows that
\begin{align*}
\tilde G(s,r) (1-e^{r+\lambda s})^2  &= H(s) (1-e^{r+\lambda s})^2  \\
&=  h(s, e^{r+\lambda s}, e^r) \left( H(0) +  \int_0^s  \lambda \tilde q(e^r, e^{r+\lambda v})  \cdot \frac{1}{h(v, e^{r+\lambda v}, e^r)} \, dv \right)   \\
&= h(s, e^{r+\lambda s}, e^r) \left( \sum_{n=m}^\infty n e^{rn} +  \int_{e^{-\lambda s}}^1  \frac{ \tilde q(e^r, e^{r+\lambda s}u) }{u h(s + (\ln u)/\lambda, e^{r+\lambda s}u, e^r)} \, du \right) .
\end{align*}
Substituting $t = s$ and $x = e^{r+\lambda s}$ we obtain:
\begin{align}
G(t,x)(1-x)^2 &= h(t, x, x e^{-\lambda t}) \left( \sum_{n=m}^\infty n (x e^{-\lambda t})^n + \int_{e^{-\lambda t}}^1 \frac{\tilde q(x e^{-\lambda t}, xu)}{u h(t+(\ln u)/\lambda, xu, x e^{-\lambda t})} du  \right) \notag \\
&=  e^{-2(\varphi(x)- \varphi(x e^{-\lambda t})) + \lambda \varphi'(1) t} (xe^{-\lambda t})^m  \left( m - (m-1) xe^{-\lambda t}  \right) \notag \\
&\hspace{5mm} +   e^{-2\varphi(x) } x^{\varphi'(1)+1}  \int_{e^{-\lambda t}}^1  e^{2\varphi(xu)}  \hat q(t + (\ln u)/\lambda, xu) \, du, \label{eq:OmegaDef}
\end{align}
%
%
where 
$$\hat q(s,v) := 2 v^{m-\varphi'(1)-1} (1-v) \sum_{i=1}^{m-1} \mathcal{R}(s, i) - 2(1-v)^2 \sum_{i=1}^{m-1} \mathcal{R}(s, i) \sum_{j=m-i}^{m-1} \overline{F}(j) v^{i+j - \varphi'(1)-1}.$$

Now take the limit as $x \nearrow 1$ to obtain:
\begin{align}
\beta(t) &:= \lim_{x \to 1} G(t,x)(1-x)^2  \label{eq:LimitG}  \\
&= e^{-2(\varphi(1)- \varphi( e^{-\lambda t})) + \lambda \varphi'(1) t} (e^{-\lambda t})^m \left( m - (m-1) e^{-\lambda t}  \right) +   e^{-2\varphi(1) }  \int_{e^{-\lambda t}}^1  e^{2\varphi(u)}  \hat q(t + (\ln u)/\lambda, u) \, du.  \notag
\end{align}

Finally, to obtain the first statement of the theorem we use a Tauberian theorem to infer the asymptotic behavior of $\mathcal{R}(t,n)$ from that of $G(t,x)$. To this end, define
$$S(t,x):=\sum_{n=m}^{\infty} \left(n - \mathcal{R}(t,n)+n \kappa\right) x^n = -G(t,x) + c \cdot \frac{(1-x) m x^m + x^{m+1}}{(1-x)^2}  ,$$
where $\kappa=(m+1)/(m-\mu+1)$, $c = \kappa + 1$ and $\mu = \varphi'(1) - 1$. From \eqref{eq:LimitG} we obtain 
$$S(t,x)\sim \left(c -\beta(t)\right)\frac{1}{(1-x)^2} \quad \text{ as } x\nearrow 1.$$
Moreover, by Lemma~\ref{MonotonicDiff} we have that $cn - \mathcal{R}(t,n) = \mathcal{K}(t,n)+n \kappa$ is monotonically increasing with $n$, so Theorem~8.3 in \cite{grandell1997mixed} yields 
$$\mathcal{K}(t,n) + n(c-1) = cn -\mathcal{R}(t,n)  \sim n \left( c - \beta(t) \right) \quad \text{ as }n\rightarrow\infty,$$
which implies that
$$ \lim_{n\rightarrow\infty} \frac{\mathcal{K}(t,n)}{n} = 1- \beta(t) = \alpha(t).$$

To obtain the asymptotic behavior of $\mathcal{C}(t,n)$ recall that $T_n^*$ denotes the time of the first request when we start with $n$ packages arranged on a circle and $B$ is the corresponding bundle size. Then, use Lemma~\ref{MonotonicDiff} to get
\begin{align}\label{eq:Cupper}
\mathcal{C}(t,n)&=E[B+\mathcal{K}(t-T_n^*,n-B)1(T_n^*<t)]\notag\\
&\leq \mu +  E[\mathcal{K}(t,n-B)]  \notag\\
&= \mu + E[\mathcal{K}(t,n-B) + \kappa (n-B)] - \kappa (n-\mu) \notag \\
&\leq \mu + \mathcal{K}(t,n) + \kappa n - \kappa(n-\mu) \qquad \text{(by Lemma~\ref{MonotonicDiff})} \notag \\
&\leq \mu + \mathcal{K}(t,n)+m\kappa.
\end{align}
Note that the expected number of packages that can be picked up during time $\Delta $, when the total number of packages is $n$, is $\lambda n \Delta  m$, so 
\begin{equation} \label{eq:TimeDiff}
\mathcal{K}(t+\Delta,n)\leq \mathcal{K}(t,n)+m\lambda n \Delta.
\end{equation}
Therefore, using Lemma \ref{MonotonicDiff} again and the observation that $E[T_n^*]=1/(\lambda n)$ (since $T_n^*$ is exponentially distributed with rate $\lambda n$), we have
\begin{align}\label{eq:Clower}
\mathcal{C}(t,n)&\geq E[\mathcal{K}((t-T_n^*)^+,n-B)]\notag\\
&= E[\mathcal{K}((t-T_n^*)^+,n-B) + (n-B)\kappa] - (n-\mu)\kappa  \notag \\
&\geq E[ \mathcal{K}((t-T_n^*)^+,n-m) + (n-m)\kappa] - (n-\mu)\kappa \qquad \text{(by Lemma~\ref{MonotonicDiff})} \notag \\
&\geq E[ \mathcal{K}(t,n-m) - m\lambda (n-m) T_n^*] - (m-\mu) \kappa  \hspace{13mm} \text{(by \eqref{eq:TimeDiff})} \notag \\
&\geq \mathcal{K}(t,n-m)-m-m\kappa.
\end{align}
Combining \eqref{eq:Cupper} and \eqref{eq:Clower} we obtain
$$\lim_{n\rightarrow\infty}\frac{\mathcal{C}(t,n)}{n}=\lim_{n\rightarrow\infty}\frac{\mathcal{K}(t,n)}{n}=\alpha(t).$$
This completes the proof. 
\end{proof}

\bigskip

It only remains to prove Theorem~\ref{convergence}, which gives the convergence rate of $\mathcal{K}(t,n)/n$ to $\alpha(t)$. The proof will be based again on the use of a Tauberian theorem, however, the monotonicity required is more difficult to verify. To ease the reading of the proof we first state a couple of preliminary technical results. 

\bigskip

\begin{proposition}\label{BoundK}
Define $c_n(t) =\mathcal{K}(t,n)-n\alpha(t)$, where $\alpha(t)$ is defined as in Theorem~\ref{alpha}. Then, for any $t \geq 0$, we have that $c_n(t)$ is either: i) bounded, ii) positive and eventually increasing, or iii) negative and eventually decreasing  in $n$.
\end{proposition}

\begin{proof}
Define $T_n$, $D(t,n,i)$ and $\hat D(t,n,j)$ as in the proof of Lemma~\ref{MonotonicDiff}, and let $\kappa = (m+1)/(m-\mu+1)$. Recall that $D(t,n,i) \geq 0$ and $\hat D(t,n,j) \geq 0$. Next, note that 
$$E[T_{n+1}] = \frac{1}{(n-\mu+2)\lambda},$$
and use equation \eqref{eq:MainIdentity}  to obtain
\begin{align}\label{convergeeq1}
&(n-\mu+2)\mathcal{K}(t,n+1)-(n-\mu+1)\mathcal{K}(t,n)\notag\\
&\geq 2\sum_{i=1}^m f(i) E\left[ \mathcal{K}((t-T_{n+1})^+, n+1-i) \right]  \notag\\
&\geq 2\sum_{i=1}^m f(i) \left( E\left[ \mathcal{K}((t-T_{n+1})^+,n+1-m) \right] - (n+1-m) \kappa \right) + 2\sum_{i=1}^m f(i) (n+1-i) \kappa   \notag\\
&= 2 E\left[ \mathcal{K}((t-T_{n+1})^+,n+1-m) \right]  + 2 m \kappa  - 2\kappa \mu  \notag \\
&\geq 2 E \left[ \mathcal{K}(t,n+1-m) - m\lambda (n+1-m) T_{n+1} \right]   \notag  \\
&\geq 2 \mathcal{K}(t,n+1-m) - 2m, 
\end{align}
where in the second inequality we used Lemma~\ref{MonotonicDiff}, and in the third inequality we used \eqref{eq:TimeDiff}. 

To obtain an upper bound note that \eqref{eq:TimeDiff} also gives
\begin{align*}
D(t,n,i) &= E\left[ \mathcal{K}((t- T_n + T_n - T_{n+1})^+, i) - \mathcal{K}((t- T_n )^+, i) \right] \\
&\leq E\left[  m\lambda i (T_n - T_{n+1}) \right] = m\lambda i \left(  \frac{1}{\lambda (n-\mu+1)} - \frac{1}{\lambda(n-\mu+2)}  \right) \\
&= \frac{mi}{(n-\mu+1)(n-\mu+2)} ,
\end{align*}
where we have used the observation that $T_n$ is exponentially distributed with rate $\lambda (n-\mu+1)$. The same arguments also yield
$$\hat D(t,n,j) \leq \frac{mj}{(n-\mu+1)(n-\mu+2)}.$$
Substituting these estimates into \eqref{eq:MainIdentity} now gives
\begin{align}
&(n-\mu+2)\mathcal{K}(t,n+1)-(n-\mu+1)\mathcal{K}(t,n) \notag \\
&\leq  \sum_{i=1}^{n-m+1} \left( \frac{m(i-1)}{(n-\mu+1)(n-\mu+2)}  + \sum_{k=1}^m f(k)  \frac{m(n-i+1-k)}{(n-\mu+1)(n-\mu+2)}   \right) \notag \\
&\hspace{5mm} + \sum_{i=1}^{m-1} F(i) \left( \frac{m(n-i)}{(n-\mu+1)(n-\mu+2)}  + \frac{mi}{(n-\mu+1)(n-\mu+2)}  \right) \notag \\
&\hspace{5mm} + 2 \sum_{i=1}^m f(i) E\left[ \mathcal{K}((t-T_{n+1})^+,n-i+1)  \right] + \mu \notag \\
&= \frac{m}{(n-\mu+1)(n-\mu+2)} \left(   (n-m+1) (n -\mu) + n (m -\mu  )   \right) \notag \\
&\hspace{5mm} + 2 \sum_{i=1}^m f(i) E\left[ \mathcal{K}((t-T_{n+1})^+,n-i+1) + \kappa (n-i+1) \right] - 2 \sum_{i=1}^m f(i) \kappa (n-i+1) + \mu \notag \\
&\leq m + \frac{n m(m-\mu)}{(n-\mu+1)(n-\mu+2)}  + 2 E\left[ \mathcal{K}((t-T_{n+1})^+,n) + \kappa n \right] - 2\kappa (n-\mu+1) +\mu \notag \\
&\leq 2 \mathcal{K}(t,n)  + 3m(\kappa+1), \label{convergeeq2}
\end{align}
where in the second inequality we used Lemma~\ref{MonotonicDiff}.  Hence, combining \eqref{convergeeq1} and \eqref{convergeeq2} we obtain that
\begin{align} \label{eq:doubleIneq}
2\mathcal{K}(t,n+1-m) - \mathcal{K}(t,n) - 2m \leq (n-\mu+2) \left( \mathcal{K}(t,n+1)- \mathcal{K}(t,n) \right) \leq   \mathcal{K}(t,n)  + 3m(\kappa+1).
\end{align}
Dividing by $n-\mu+2$, taking the limit as $n \to \infty$, and using Theorem~\ref{alpha}, we obtain
\begin{align*}
\alpha(t) = \lim_{n \to \infty} \frac{2\mathcal{K}(t,n+1-m) - \mathcal{K}(t,n)}{n-\mu+2} &\leq \lim_{n \to \infty} \left(  \mathcal{K}(t,n+1)- \mathcal{K}(t,n) \right) \leq \lim_{n \to \infty} \frac{\mathcal{K}(t,n)}{n-\mu+2} = \alpha(t). 
\end{align*}

This in turn implies that
\begin{equation} \label{eq:LimitCs}
\lim_{n\rightarrow\infty} c_{n+1}(t) - c_n(t)=0
\end{equation}
for all $t \geq 0$. Furthermore, \eqref{eq:doubleIneq} can be written as:
$$\frac{2 c_{n+1-m}(t) - c_n(t) - \alpha(t)(2m-\mu) - 2m}{n-\mu+2}  \leq  c_{n+1}(t) - c_n(t)  \leq \frac{c_n(t) + \alpha(t) (\mu-2) + 3m(\kappa+1)}{n-\mu+2}.$$
which combined with \eqref{eq:LimitCs} gives that  there exists $n_0 \in \mathbb{N}_+$ such that 
$$c_{n-m+1}(t) \geq c_{n}(t) -m/2\text{ for all }n \geq n_0,$$
which implies that
$$\frac{ c_{n}(t) - \alpha(t)(2m-\mu) - 3m}{n-\mu+2}  \leq  c_{n+1}(t) - c_n(t)  \leq \frac{c_n(t) + \alpha(t) (\mu-2) + 3m(\kappa+1)}{n-\mu+2}$$
for all $n \geq n_0$. 

To complete the proof, note that if there exists an $n' \geq n_0$ such that $c_{n'}(t) + \alpha(t)(\mu-2) + 3m(\kappa+1) < 0$, then $c_{n'}(t) < 0$ and $c_{n'+1}(t) < c_{n'}(t)$, which in turn implies that,
$$c_{n'+1}(t) + \alpha(t)(\mu-2)+3m(\kappa+1)<c_{n'}(t)+ \alpha(t)(\mu-2)+3m(\kappa+1)<0.$$
Iterating in this way we obtain that
$$c_{n'+k}(t) < c_{n'+k-1}(t) < \cdots < c_{n'}(t) < 0$$
for all $k \geq 1$. This gives condition (iii) of the proposition. 

Suppose now that there exists $n' \geq n_0$ such that $c_{n'}(t) - \alpha(t) (2m-\mu) - 3m > 0$. Similarly as above, we obtain that $c_{n'}(t) > 0$ and $c_{n'+1}(t) > c_{n'}(t)$, which in turn implies that,
$$c_{n'+1}(t) -  \alpha(t) (2m-\mu) - 3m > c_{n'}(t) - \alpha(t) (2m-\mu) - 3m > 0.$$
Iterating as before, we obtain that
$$c_{n'+k}(t) > c_{n'+k-1}(t) > \cdots > c_{n'}(t) > 0$$
for all $k \geq 1$. This gives condition (ii) of the proposition.

Finally, neither of the two previous cases occurs we have that
$$-\alpha(t)(\mu-2)-3m(\kappa+1)\leq c_n(t)\leq \alpha(t)(2m-\mu)+3m$$
for all $n \geq n_0$, and condition (i) follows. 
\end{proof}

\bigskip

The second preliminary result provides the differentiability of the functions that determine $\alpha(t)$.

\begin{lemma} \label{L.Differentiability}
Define
\begin{align*}
\omega_t(x) &:=   e^{-2(\varphi(x)- \varphi(x e^{-\lambda t})) - \lambda t (m-\varphi'(1)) } x^m  \left( m - (m-1) xe^{-\lambda t}  \right) \\
&\hspace{5mm} +   e^{-2\varphi(x) } x^{\varphi'(1)+1}  \int_{e^{-\lambda t}}^1  e^{2\varphi(xu)}  \hat q(t + (\ln u)/\lambda, xu) \, du,
\end{align*}
where 
$$\hat q(s,v) := 2 v^{m-\varphi'(1)-1} (1-v) \sum_{i=1}^{m-1} \mathcal{R}(s, i) - 2(1-v)^2 \sum_{i=1}^{m-1} \mathcal{R}(s, i) \sum_{j=m-i}^{m-1} \overline{F}(j) v^{i+j - \varphi'(1)-1}.$$
Then, for any fixed $t \geq 0$, $\omega_t(x)$ is infinitely differentiable on $[0, \infty)$ and 
$$\sup_{0 \leq x \leq 1} |\omega_t''(x)| \leq H_F$$
for some constant $H_F < \infty$ that depends only on the distribution $F$. 
\end{lemma}

\begin{proof}
Start by defining 
\begin{align}
p(s,v) &:= \hat q(s,v) v^{\varphi'(1)+1-m} = 2(1-v) \sum_{i=1}^{m-1} \mathcal{R}(s,i) - 2(1-v)^2 \sum_{i=1}^{m-1} \mathcal{R}(s, i) \sum_{k=0}^{i-1} \overline{F}(k-i+m) v^{k} \notag \\
&=  2(1-v) \sum_{i=1}^{m-1} \mathcal{R}(s,i) - 2(1-v)^2 \sum_{k=0}^{m-2} v^k \sum_{i=k+1}^{m-1} \mathcal{R}(s, i) \overline{F}(k-i+m) =: \sum_{k=0}^m c_k(s) v^k. \label{eq:pDef}
\end{align}
Note that $p(s,v)$ is a polynomial of order $m$ in $v$, and $x^{\varphi'(1)+1} \hat q(s,xu) = x^m u^{\varphi'(1)+1-m} p(s,xu)$, from where it follows that $x^{\varphi'(1)+1} \hat q(s,xu)$ is infinitely differentiable in $x$ on $[0, \infty)$. Since $\varphi(y) = \sum_{i=1}^{m-1} \overline{F}(i) y^i/i$ is also infinitely differentiable on the real line, then $\omega_t(x)$ is infinitely differentiable in $x$ on $(0,\infty)$. Moreover, by writing
\begin{align*}
\omega_t(x) &= a(t) e^{2(\varphi(x e^{-\lambda t}) - \varphi(x ))} x^m - b(t) e^{2(\varphi(x e^{-\lambda t}) - \varphi(x ))} x^{m+1} \\
&\hspace{5mm} + \int_{e^{-\lambda t}}^1 e^{2(\varphi(xu) - \varphi(x))} \sum_{k=0}^{m} c_k(t+(\ln u)/\lambda)   u^{m-\varphi'(1)-1+k} x^{m+k} du,
\end{align*}
with $a(t) := m e^{-\lambda t(m-\varphi'(1))}$ and $b(t) := (m-1) e^{-\lambda t(m-\varphi'(1)+1)}$, we obtain that
\begin{align*}
\omega_t''(x) &= a(t) \frac{\partial^2}{\partial x^2}  e^{2(\varphi(x e^{-\lambda t}) - \varphi(x ))} x^m - b(t) \frac{\partial^2}{\partial x^2}  e^{2(\varphi(x e^{-\lambda t}) - \varphi(x ))} x^{m+1} \\
&\hspace{5mm} + \sum_{k=0}^m \int_{e^{-\lambda t}}^1  c_k(t+(\ln u)/\lambda) u^{m-\varphi'(1)-1+k} \frac{\partial^2}{\partial x^2}   e^{2(\varphi(x u) - \varphi(x ))} x^{m+k} du \\
&= a(t) e^{2(\varphi(x e^{-\lambda t}) - \varphi(x ))} \nu_m(x,e^{-\lambda t}) - b(t) e^{2(\varphi(x e^{-\lambda t}) - \varphi(x ))} \nu_{m+1}(x,e^{-\lambda t}) \\
&\hspace{5mm} + \sum_{k=0}^m \int_{e^{-\lambda t}}^1 c_k(t+(\ln u)/\lambda) u^{m-\varphi'(1)-1+k} e^{2(\varphi(xu) - \varphi(x))}\nu_{m+k}(x,u) \, du,
\end{align*}
where 
\begin{align*}
\nu_k(x,u) &= 4(\varphi'(xu) u - \varphi'(x))^2 x^k + 4(\varphi'(xu) u - \varphi'(x)) k x^{k-1} 1(k \geq 1) \\
&\hspace{5mm} + 2(\varphi''(xu) u^2 - \varphi''(x)) x^k + k(k-1) x^{k-2} 1(k \geq 2).
\end{align*}

Noting that $\varphi(y), \varphi'(y)$, and $\varphi''(y)$ are all increasing in $y$, gives that $e^{2(\varphi(xu) - \varphi(x))} \leq 1$ and 
\begin{align*}
|\nu_k(x,u)| &\leq 4 \varphi'(x)^2 x^k + 4 \varphi'(x) k x^{k-1} 1(k\geq 1) + 2\varphi''(x) x^k + k(k-1) x^{k-2} 1(k \geq 2) \\
&\leq 4\varphi'(1)^2 + 4\varphi'(1)k + 2\varphi''(1) + k(k-1) =: H_k,
\end{align*}
for all $x,u \in [0,1]$, which in turn implies that for any $t \geq 0$, 
\begin{align*}
|\omega_t''(x)| &\leq a(t) | \nu_m(x,e^{-\lambda t}) | + b(t) | \nu_{m+1}(x, e^{-\lambda t})| + \sum_{k=0}^m \int_{e^{-\lambda t}}^1 |c_k(t+(\ln u)/\lambda) \nu_{m+i}(x,u)| \, du \\
&\leq a(t) H_m  + b(t) H_{m+1}  + \sum_{k=0}^m H_{m+k} \int_{e^{-\lambda t}}^1 |c_k(t+(\ln u)/\lambda)| \, du .
\end{align*}
To complete the proof note that $a(t) \leq m$ and $b(t) \leq m$ for all $t \geq 0$, and from \eqref{eq:pDef} it can be verified that:
$$|c_k(s)| \leq 2 \sum_{i= (k-1)\vee 1}^{m-1} \mathcal{R}(s,i) \leq 2 \sum_{i=(k-1)\vee 1}^{m-1} i, \qquad 0 \leq k \leq m,$$
from where we obtain that
$$|\omega_t''(x)| \leq mH_m + m H_{m+1} + 2 \sum_{k=0}^m H_{m+k} \sum_{i=(k-1)\vee 1}^{m-1} i =: H_F$$
for all $t \geq 0$ and all $x \in [0,1]$. 
\end{proof}

\bigskip

We can now give the proof of Theorem~\ref{convergence}. 

\bigskip

\begin{proof}[Proof of Theorem \ref{convergence}]
Fix $t$ and let $\alpha(t)=\alpha(t,\lambda)$. Define $J(t,x):=\sum_{n=m}^\infty (\mathcal{K}(t,n) - n\alpha(t)) x^n$. Recall from the proof of Theorem~\ref{alpha} that $G(t,x) = \sum_{n=m}^\infty \mathcal{R}(t,n) x^n$, and note that from \eqref{eq:OmegaDef} we have that
\begin{align*}
G(t,x) (1-x)^2 &= \omega_t(x),
\end{align*}
where $\omega_t(x)$ is defined in Lemma~\ref{L.Differentiability}. Moreover, by Lemma~\ref{L.Differentiability} we have that $\omega_t(x)$ is infinitely differentiable on $(0, \infty)$ and satisfies $\sup_{0 \leq x \leq 1} |\omega_t''(x)| \leq H_F$, for some constant $H_F < \infty$, independent of $t$. Also, by \eqref{eq:LimitG}, we have $\omega_t(1) = \beta(t)$. Next, write $J(t,x)$ as:
\begin{align*}
J(t,x) &= - G(t,x) + \beta(t) \sum_{n=m}^\infty n x^n = \frac{\beta(t) - (1-x)^2 G(t,x) }{(1-x)^2} + \beta(t) \left(\sum_{n=m}^\infty n x^n - (1-x)^{-2}\right)\\
& = \frac{\beta(t) - \omega_t(x)}{(1-x)^2}+\beta(t)\frac{(1-m)x^{m+1}+mx^m-1}{(1-x)^2}.
\end{align*}
Since 
$$\lim_{x\rightarrow 1}\frac{(1-m)x^{m+1}+mx^m-1}{1-x}=-1,$$
and 
$$\omega_t(x)=\beta(t)+\omega_t'(1)(x-1)+O\left( (x-1)^2 \right) \text{ as }x\nearrow 1,$$
we have 
\begin{align*}
\lim_{x\nearrow 1}(1-x)J(t,x)&=\lim_{x\nearrow 1}\frac{\beta(t)-\omega_t(x)}{1-x}+\beta(t)\frac{(1-m)x^{m+1}+mx^m-1}{1-x}\\
&=\omega_t'(1)-\beta(t).
\end{align*}

Therefore, we have
\begin{align*}
\sum_{n=0}^\infty \left(\mathcal{K}(t,n)-n\alpha(t)\right)x^n\sim \frac{\omega_t'(1)-\beta(t)}{1-x} \text{ as } x\uparrow 1.
\end{align*}
From Proposition~\ref{BoundK}, $\mathcal{K}(t,n)-n\alpha(t)$ is either bounded, or eventually negative and decreasing, or eventually positive and increasing. If $\mathcal{K}(t,n)-n\alpha(t)$ is bounded, then 
$$\left|\frac{\mathcal{K}(t,n)}{n}-\alpha(t)\right|=O(n^{-1}) \text{ as } n\rightarrow\infty,$$
where we finishes the proof.
If $\mathcal{K}(t,n)-n\alpha(t)$ is eventually negative and decreasing, or eventually positive and increasing, then from Tauberian Theorem (Theorem 8.3 in~\cite{mcfadden1965mixed}), we have
$$\mathcal{K}(t,n)-n\alpha(t)\sim\frac{\omega_t'(1)-\beta(t)}{\Gamma(1)}.$$
That is, 
$$\left|\frac{\mathcal{K}(t,n)}{n}-\alpha(t)\right|=O(n^{-1}) \text{ as } n\rightarrow\infty.$$

To complete the proof, use inequalities \eqref{eq:Cupper} and \eqref{eq:Clower} to obtain that
\begin{align*}
\mathcal{K}(t,n-m)-m-m\kappa \leq \mathcal{C}(t,n)\leq \mu + \mathcal{K}(t,n)+m\kappa,
\end{align*}
where $\kappa = (m+1)/(m-\mu+1)$.
and conclude that
$$\left|\frac{\mathcal{C}(t,n)}{n}-\alpha(t)\right|=O(n^{-1}) \text{ as } n\rightarrow\infty.$$
\end{proof}

\bigskip

The last proof in the paper is that of Lemma~\ref{L.CompareStrategies}, which compares the expected cost of of using a strategy which combines private drivers with in-house vans with the strategy of using only vans. 

\bigskip

\begin{proof}[Proof of Lemma~\ref{L.CompareStrategies}]
To start, use Theorem~\ref{T.Daganzo} to obtain that for any $z$ and $\lambda := \lambda(z)$ we have
\begin{align*}
&\text{Cost}_{P}(z; {\bf x}^{(n)}) - \text{Cost}_V({\bf x}^{(n)}) \\
&=   \mathcal{C}(T,n,\lambda) \left(\zeta_P+\frac{h_P+z}{v_P}\right) \left(\frac{ \overline{r}(n) }{E[B]} + \frac{1}{n} L(\text{TSP}({\bf x}^{(n)}))\right) \\
&\hspace{5mm} +\mathcal{C}(T,n,\lambda)  \left((h_P+z)\tau_P - h_V\tau_V  \right) \\
&\hspace{5mm} + \left(\zeta_V+\frac{h_V}{v_V}\right) \left\{  \mathbb{E}_n\left[ L(\text{CVRP}({\bf Y}^{(n)}(\lambda^*) )) \right]   -  L(\text{CVRP}({\bf x}^{(n)}) )) \right\}  \\
&\leq \mathcal{C}(T,n,\lambda) \left( \left(\zeta_P+\frac{h_P+z}{v_P}\right) \frac{r^*}{E[B]}   + (h_P + z)\tau_P - h_V\tau_V \right) \\
&\hspace{5mm} +  \mathcal{C}(T,n,\lambda) \left(\zeta_P+\frac{h_P+z}{v_P}\right) \left( \frac{\overline{r}(n) - r^*}{E[B]} +  \frac{1}{n} L(\text{TSP}({\bf x}^{(n)}))\right) \\
&\hspace{5mm} + \left(\zeta_V+\frac{h_V}{v_V}\right)  \left\{ \mathbb{E}_n\left[ 2\left( \frac{n-\tilde N(T,n,\lambda)}{V} + 1 \right) \mathscr{R}(n-\tilde N(T,n,\lambda)) + L(\text{TSP}({\bf Y}^{(n)}(\lambda))) \right]  \right. \\
&\hspace{5mm} \left.  - \max\left\{   \frac{2n \overline{r}(n)}{V}, \, L(\text{TSP}({\bf x}^{(n)})) \right\}   \right\} ,
\end{align*}
where $\mathscr{R}(n-\tilde N(T,n,\lambda)) = (n- \tilde N(T,n,\lambda))^{-1} \sum_{i=1}^{n - \tilde N(T,n,\lambda)} d({\bf Y}_i , \boldsymbol{0}) $. Next, divide by $n$, take the limit as $n \to \infty$, and use Theorem~\ref{alpha} to obtain that
\begin{align*}
&\limsup_{n \to \infty} \frac{1}{n} \left( \text{Cost}_{P}(z; {\bf x}^{(n)}) - \text{Cost}_V({\bf x}^{(n)}) \right) \\
&\leq \alpha(T,\lambda) \left( \left(\zeta_P+\frac{h_P+z}{v_P}\right) \frac{r^*}{E[B]}   + (h_P + z)\tau_P - h_V\tau_V \right) + \alpha(T,\lambda) \limsup_{n \to \infty} \frac{1}{n} L(\text{TSP}({\bf x}^{(n)}) ) \\
&\hspace{5mm} + \left(\zeta_V+\frac{h_V}{v_V}\right)  \left\{ \limsup_{n \to \infty} \mathbb{E}_n\left[ 2\left( \frac{n-\tilde N(T,n,\lambda)}{V n} + \frac{1}{n} \right) \mathscr{R}(n-N(T,n,\lambda)) \right]  \right. \\
&\hspace{5mm} \left.  + \limsup_{n \to \infty}  \mathbb{E}_n\left[  \frac{L(\text{TSP}({\bf Y}^{(n)}(\lambda)))}{n}   \right] - \max\left\{   \frac{2 r^*}{V}, \, \liminf_{n \to \infty} \frac{1}{n} L(\text{TSP}({\bf x}^{(n)})) \right\}   \right\} .
\end{align*}
Now use Theorem~\ref{T.TSPbound} to obtain that 
\begin{align*}
0 \leq \limsup_{n \to \infty}  \mathbb{E}_n\left[  \frac{L(\text{TSP}({\bf Y}^{(n)}(\lambda)))}{n}   \right]  &\leq \limsup_{n \to \infty}  \frac{L(\text{TSP}({\bf x}^{(n)}))}{n} = 0 \qquad \text{a.s.}
\end{align*}
Moreover, since packages are arranged on a circle at the beginning of the day, they are all equally likely to be undelivered by time $T$, and their specific location is independent of $\tilde N(T,n,\lambda)$ (note that the $\{{\bf Y}_i\}$ may not be independent though). Hence, $\mathbb{E}_n[ d({\bf Y}_i, {\bf 0}) | \tilde N(T,n,\lambda) ] = \overline{r}(n)$, which combined with  Theorem~\ref{convergence} gives
\begin{align*}
& \limsup_{n \to \infty} \mathbb{E}_n\left[ 2\left( \frac{n-\tilde N(T,n,\lambda)}{V n} + \frac{1}{n} \right) \mathscr{R}(n-\tilde N(T,n,\lambda)) \right]  \\
&= \frac{2}{V}  \limsup_{n \to \infty} \frac{1}{n} \mathbb{E}_n\left[  \sum_{i=1}^{n- \tilde N(T,n,\lambda)} d({\bf Y}_i, {\bf 0}) \right]  + 2 \limsup_{n \to \infty} \frac{1}{n} \mathbb{E}_n\left[ \mathscr{R}(n-\tilde N(T,n,\lambda)) \right]  \\
&\leq  \frac{2}{V} \lim_{n \to \infty} \frac{n-\mathcal{C}(T,n,\lambda)}{n} \cdot \overline{r}(n) +  \lim_{n \to \infty} \frac{\rho}{n}  \\
&= \frac{2 (1-\alpha(T,\lambda)) r^*}{V} ,
\end{align*}
where $\rho := \sup_{{\bf x} \in R} d({\bf x}, {\bf 0})$ is the radius of the bounded region $R \subseteq \mathbb{R}^2$. 

We conclude that
\begin{align*}
&\limsup_{n \to \infty} \frac{1}{n} \left( \text{Cost}_{P}(z; {\bf x}^{(n)}) - \text{Cost}_V({\bf x}^{(n)}) \right)  \\
&\leq \alpha(T,\lambda)  \left( \left(\zeta_P+\frac{h_P+z}{v_P}\right) \frac{r^*}{E[B]}   + (h_P + z )\tau_P- h_V \tau_V \right) + \left( \zeta_V + \frac{h_V}{v_V} \right) \left\{  \frac{2(1-\alpha(T,\lambda)) r^*}{V} - \frac{2r^*}{V} \right\} \\
&= - \alpha(T,\lambda) \left(  \left(\zeta_V+\frac{h_V}{v_V}\right) \frac{2r^*}{V} -  \left(\zeta_P+\frac{h_P}{v_P}\right) \frac{r^*}{E[B]}  - (h_P\tau_P - h_V\tau_V) - z \left( \frac{r^*}{v_P E[B]} + \tau_P   \right)  \right) .
\end{align*}
\end{proof}

\section{Numerical experiments} \label{S.Numerical}

This section provides numerical examples illustrating the computational effort of implementing our proposed strategy, the calibration of the parameters, as well as a cost-benefit analysis. Specifically, we focus on the following:

\begin{itemize} 
\item Choosing the parameters: in particular, how to estimate the costs needed for the formulation of the optimization problem in Section~\ref{S.Pricing}, for which we cite empirical studies in the transportation literature.
\item Numerical case studies: including a full implementation of our proposed strategy under two different choices for the distribution of package destinations. 
\item Cost-benefit analysis: a comparison between the proposed strategy that uses a combination of private drivers with an in-house van delivery system, and a more traditional van-only strategy. 
\item Justification of the mathematical assumptions: in particular, of how distances between locations along an optimal TSP tour are typically small, which allows us to assume that all bundles of the same size are equally desirable. 
\end{itemize}

\subsection{Parameters} 

The parameter estimation mostly follows the sources used in \cite{qi2017shared}.

{\em Parameters related to vans:} The parameters that need to be estimated are:  $\zeta_V$, the per-mile transportation cost; $h_V$, the opportunity cost for van drivers; $v_V$, the average speed of a delivery van; $V$, the van capacity; and $\tau_V$, the end-point delivery time for vans. To compute $\zeta_V$ we use the average diesel price in the U.S. for 2017 (taken from the web), and the estimated maintenance and depreciation costs used in \cite{lammert2009twelve} and \cite{barnes2003per}, respectively. In particular, the average diesel price in the U.S. for 2017 was \$2.540 per gallon and the fuel efficiency of a UPS van is 10.6 miles per gallon of diesel, which gives a fuel cost for vans of \$0.25 per mile. The maintenance cost for vans was computed in \cite{lammert2009twelve} to be \$0.152 per mile in 2009, and the depreciation cost of a van under city driving conditions was computed in \cite{barnes2003per} to be \$0.081 per mile in 2003. Adjusting for an annual inflation of 2.5\%, we compute a van's per mile cost to be:
$$\zeta_V = 0.25+0.152\times (1.025)^8+0.081\times (1.025)^{14}=0.550.$$
 To estimate $h_V$, we use that a van driver's wage was approximately \$30 per hour~\cite{barnes2003per} in 2003, so the inflation-adjusted wage in 2017 is 
 $$h_V = 30 \times (1.025)^{14}=42.389.$$
The average speed of a UPS van was estimated in \cite{lammert2009twelve} to be $v_V = 24.1$ miles/hour. We use an average van capacity of $V = 200$ packages as in \cite{qi2017shared}, which corresponds to a van capacity of 2,000 kg and an average package weight of 10 kg. Finally, for the end-point delivery time we use $\tau_V = 97$ seconds as estimated in \cite{qi2017shared}.

\bigskip
{\em Parameters related to private cars:} The parameters related to private cars that need to be estimated are: $\zeta_P$, the per-mile transportation cost; $v_P$, the average speed; and $h_P$, the opportunity cost for private drivers (Uber/Lyft drivers). For the end-point delivery time for drivers we simply use $\tau_P = \tau_V$.   To estimate $\zeta_P$ we use the estimated hourly expenses for part-time Uber drivers computed in Table~6 of \cite{hall2015analysis}, which after taking the average over vehicle types gives \$3.84 per hour. Also, we use an average speed for cars in U.S. cities of $v_P = 29.9$ miles per hour (http://infinitemonkeycorps.net/projects/cityspeed/), which yields a per mile cost of
$$\zeta_P = \frac{3.84}{29.9} = 0.1284.$$
To estimate $h_P$ note that the average gross income per hour of an Uber driver is given by $h_P + \zeta_P v_P$, which in \cite{hall2015analysis} was estimated to be \$19.35 per hour in 2015. After adjusting for inflation we obtain that $h_P + \zeta_P v_P = 19.35 \times (1.025)^2 = 20.33$, and therefore, the per hour opportunity cost is
$$h_P = 20.33 - 0.1284 \times 29.9 = 16.49.$$

\begin{remark}
Using the parameters for the vans and private drivers discussed above, we have $\zeta_P+ h_P/v_P = 0.68$, $h_P=16.49$, $\zeta_V+\frac{h_V}{v_V}=2.31$, $h_V=42.389$, and $\tau_V = \tau_P = 0.0269$ ($= 97$s). Note that all of these were estimated for an average U.S. city, and are therefore fairly robust to where the strategy is deployed. However, the values of average distance to the depot, $r^*$, and the mean bundle size, $E[B]$, are strongly related to the specific region, and will significantly impact the amount of possible improvement  with the use of the new strategy. In terms of $r^*$ and $E[B]$, the sufficient condition  in Lemma~\ref{L.CompareStrategies} ensuring that our mixed strategy is better than the van-only strategy becomes:
$$0.023r^*-0.635\frac{r^*}{E[B]}+0.698>0.$$
\end{remark}

\bigskip
{\em Arrival rate and bundle size distribution:} Since we were unable to find any references quantifying the relationship between incentives and the supply of private drivers, as well as for the preferences on the number of packages that private drivers may want to deliver, our choices of $\lambda(z)$ and $F$ are somewhat arbitrary. For the first one we assume a linear relationship of the form $\lambda(z) = b + a z$, and estimate $a$ and $b$ using data for Uber and Lyft drivers that has been collected for transporting people (not packages). Let $s(w) = b' + a' w$ denote the supply of Uber drivers in the San Francisco area (about 50 square miles) when the average revenue per unit of time is $w$.  To fit the values of $(a', b')$ we found at http://www.govtech.com/question-of-the-day/Question-of-the-Day-for-061917.html that ride-sharing companies make about 170,000 trips on an average weekday in the San Francisco area (about 50 square miles) and Uber's market share was about 3 times that of Lyft in 2017 (https://www.recode.net/2017/8/31/16227670/uber-lyft-market-share-deleteuber-decline-users), so if we assume that the active operation time is 18 hours, then Uber's average number of trips per hour in the San Francisco area is 7083. Using as a base revenue of $w_0 = 20.33$ dollars per hour, we obtain that $s(w_0) = 7083$. According to the work done by Hall et al.~\cite{hall2015effects}, the supply of Uber drivers doubles when the surge pricing increases the wage per hour from $w_0$ to $1.8 w_0$, which gives $s(1.8 w_0) = 14166$, and we obtain $a' = 435.50$ and $b' = -1770.75$.  Since the wage $w$ and the incentive rate $z$ are linearly related ($w = z+h_P + \zeta_P v_P$), it follows from the linearity of $\lambda(z)$ that the total number of Uber drivers willing to deliver packages instead of transporting passengers is a linear function of $s(w)$. Furthermore, we assume that at the base revenue per hour $w_0$, only 10\% of the Uber drivers would be interested in delivering packages\footnote{At this time, we have no way to estimate this 10\% since we are unaware of any data for this type of business.}, which gives $n\lambda(0)  = 4000 \lambda(0) = 0.1 s(w_0)$, while at the surge pricing rate $1.8 w_0$ we expect almost all Uber drivers to prefer the delivery of packages, which yields $n\lambda(1.8 w_0 - h_P - \zeta_P v_P) = n \lambda(16.26 ) =  s(1.8 w_0 )$. To estimate $n$, we used the fact that there are around 356,916 households in San Francisco (https://datausa.io/profile/geo/san-francisco-ca/\#housing), and the average online shopping frequency in 2017 was around 21 times per year (https://www.statista.com/statistics/448659/online-shopping-frequency-usa/), so we use $n= 356916 \times 21/365 = 20535$. Hence, we obtain 
$$\lambda(z) = 0.03 +  0.04 z.$$

For the distribution of bundles, we assume that $F$ is has the following distribution: 
$$P(B = k) = \left( \sum_{i=1}^{20} e^{-10} (10)^i/i!  \right)^{-1} \frac{e^{-10} (10)^k}{k!}, \qquad k = 1, 2,  \dots, 20,$$
which corresponds to a Poisson distribution with mean 10 packages conditioned on taking values on $\{1, 2, \dots, 20\}$.

Table \ref{T.parameters} lists the parameters we use in both of the numerical experiments in the following section. We target on the region with size 5 miles $\times$ 5 miles with total number of packages $n=2000$. 

\begin{table}[H]  
\centering 
\caption{Parameters}  \label{T.parameters}
 \begin{tabular}{lll}  
     \hline
    {\em  Parameter} & {\em Value} \\
     \hline
       Total number of packages $n$ & 2000\\
       Bundle size distribution $F$ & Poisson(10), conditioned in \{1, 2, \dots, 20\} \\
       Requests arrival rate $\lambda$ & $\lambda=0.03 + 0.04z$\\
       Delivery time window $[0, T]$ & 8h\\
       Area of the square region $A$ & 25 square miles\\
       \hline
   \end{tabular}
\end{table}

\begin{figure}[t]
\begin{center}
  \includegraphics[width=8cm]{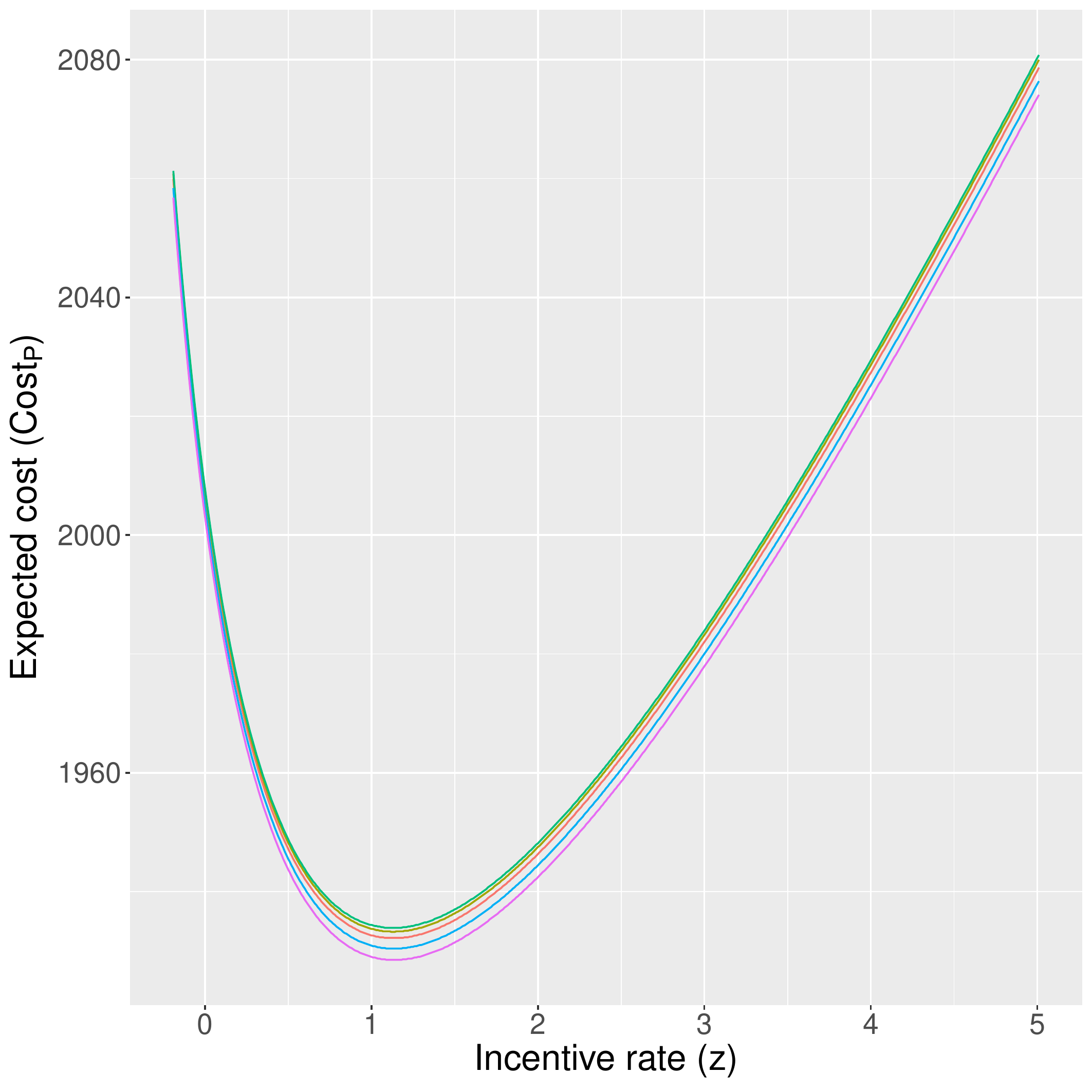}
  \caption{Objective function \eqref{eq:Optimization} for 2000 packages uniformly distributed over a square region.  Different colors represent different realizations of package destinations.}
\label{F.obj_uniform}
  \end{center}
\end{figure}

\subsection{Case studies}

To illustrate the impact that the distribution of the package destinations over the service region may have, we consider two different cases, one where the destinations are uniformly distributed throughout a square region, and another one where we can identify three different clusters in the same square region. For both cases we use a total of 2000 packages, and we sample independently their destinations using stochastic simulation. We then solve \eqref{eq:Optimization} and compute the optimal incentive rate $z^*$. To solve the TSP we used Concorde implemented in R, and for the CVRP we used VRPH 1.0.0.  Throughout all our computations we use the $L_1$ distance.

\subsubsection{Case 1: Packages distributed uniformly over a square region}

Note that the objective function \eqref{eq:Optimization} depends on the specific package destinations \linebreak ${\bf x}^{(n)} = \{{\bf x}_1, \dots, {\bf x}_n\}$ only through the length of the TSP route, L(TSP), and the average long-haul distance $\overline{r}(n)$. For $n = 2000$, the differences from one realization to another are very small, as shown in Figure~\ref{F.obj_uniform}.

\begin{figure}[t]
\centering
\begin{subfigure}[h]{.48\textwidth}
\centering
\includegraphics[width=\linewidth,height=6cm]{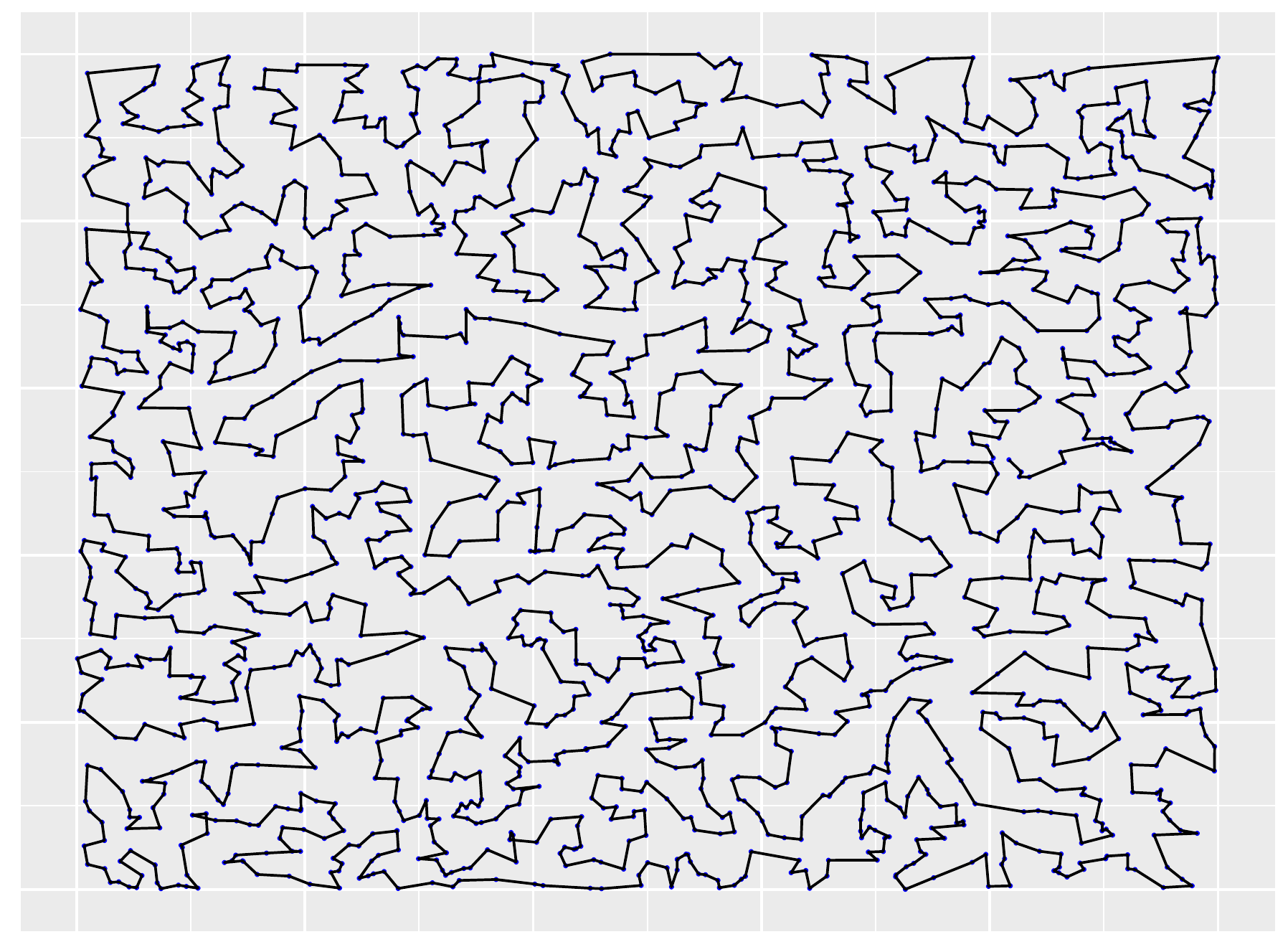}
\caption{Optimal TSP route solved using Concorde in less than 3 seconds. Length = 207.81 miles. This computation is needed for the mixed strategy.}
 \label{F.tsp_uniform}
\end{subfigure}
\begin{subfigure}[h]{.48\textwidth}
\centering
\includegraphics[width=\linewidth,height=6cm]{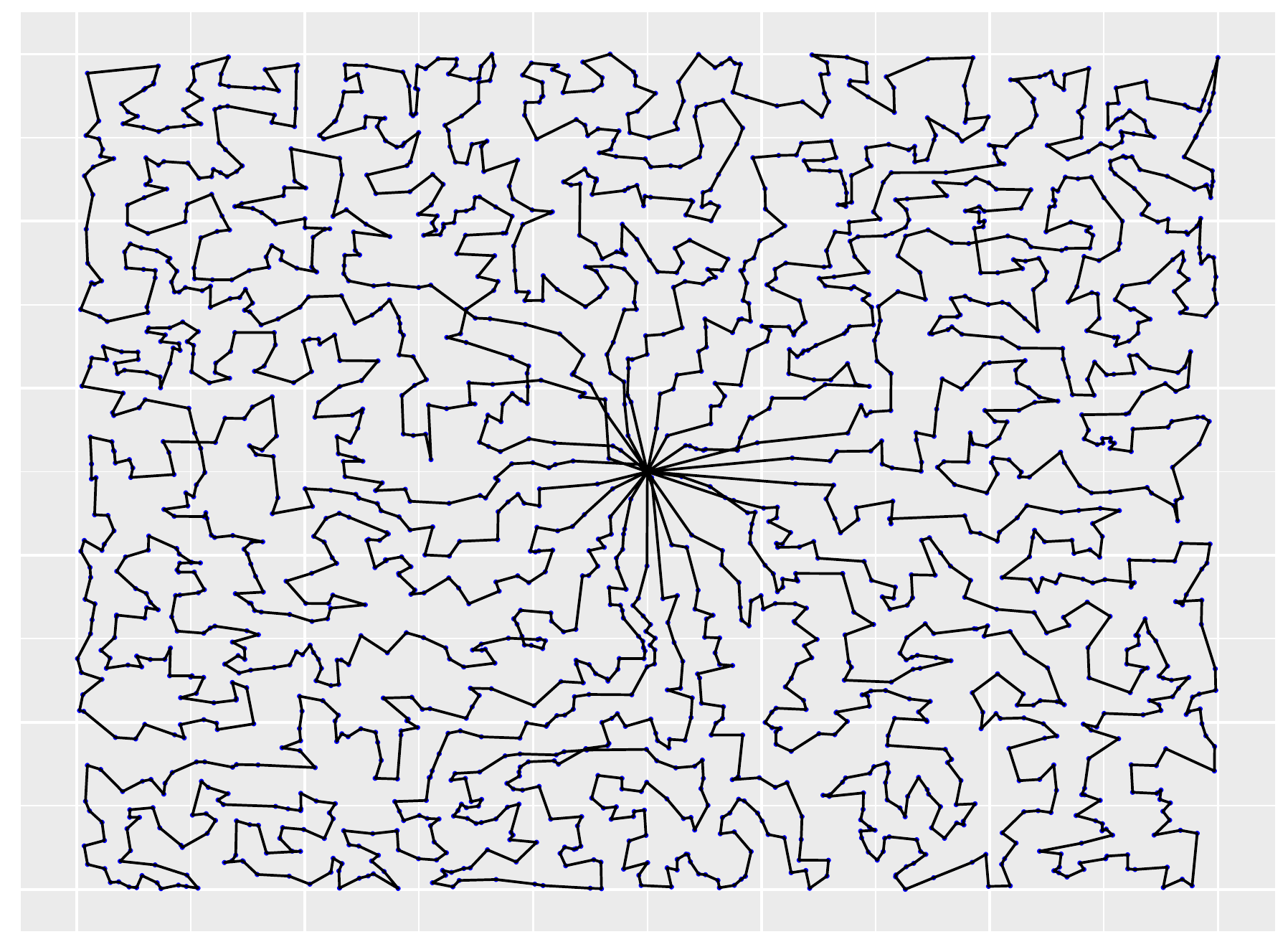}
\caption{Optimal CVRP route solved using VRPH 1.0.0 in less than 60 seconds. Length = 222.83 miles. This computation is needed for the van-only strategy.}\label{F.vrp_uniform}
\end{subfigure}
\caption{Optimal TSP and CVRP routes for 2000 packages uniformly distributed on a $5\times5$ square miles region.}
\end{figure}

Next, we used stochastic simulation to generate a single realization of the package destinations ${\bf x}^{(n)}$ along with its corresponding package pick-up process. More precisely, we simulated the $n = 2000$ independent Poisson processes with rate $\lambda(z^*)$, where $z^* = 1.11$ solves \eqref{eq:Optimization}, during the time window $[0, T]$, $T = 8$ hours. We then computed an optimal CVRP route on the leftover packages. To compare the new mixed strategy with the van-only strategy we also computed the length of an optimal CVRP route on the original 2000 packages.  Figure~\ref{F.tsp_uniform} shows an optimal TSP route of length L(TSP)$= 207.81$ miles, solved using Concorde, while Figure~\ref{F.vrp_uniform} shows an optimal CVRP route of length L(CVRP)$=222.83$ miles, solved using VRPH 1.0.0. Both routes are for the original 2000 packages, with the latter needed for the van-only strategy. 

Using the simulated CVRP route from Figure~\ref{F.vrp_uniform} we computed the cost of the van-only strategy to be:
$$\text{L(CVRP(${\bf x}^{(n)}$))} \left( \zeta_V + \frac{h_V}{v_V}  \right)  + nh_V \tau_V =  222.83 (2.309) +2000 (42.389) (97/3600) = \$ 2798.81.$$

For our proposed mixed strategy, we computed an optimal CVRP route for the leftover packages at time $T$, which in the simulation run we did had 554 packages and a length of 96.07 miles (see Figure~\ref{F.left_uniform}). We then computed the total cost for delivering the 2000 packages, which was computed to be \$1888.26. The improvement compared to the van-only strategy was 32.53\%.

\begin{figure}[h]
\begin{center}
 \includegraphics[width=8cm,height = 6cm]{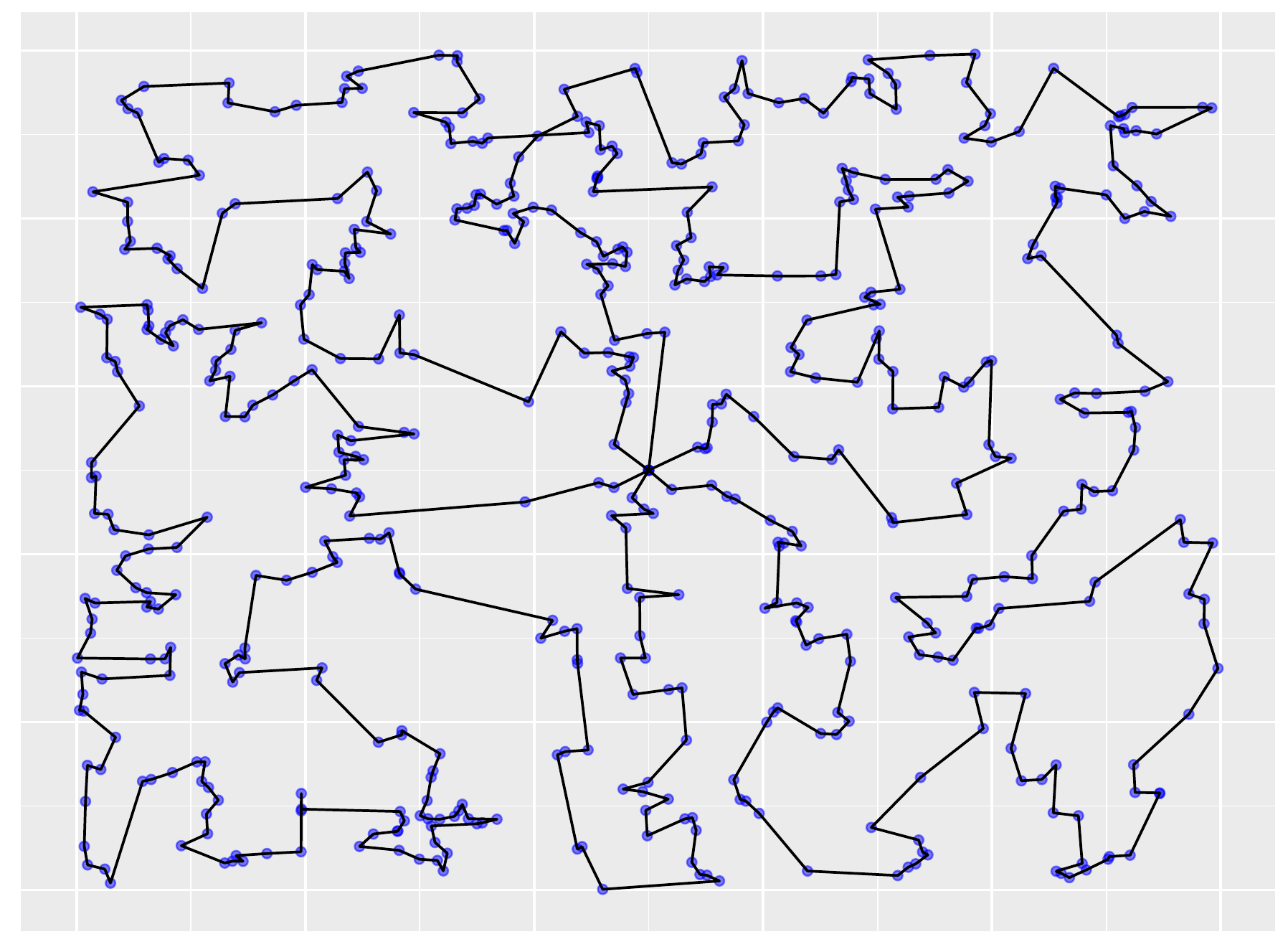}
  \caption{Optimal CVRP route for 554 leftover packages solved using VRPH 1.0.0. Length = 96.07 miles.}
\label{F.left_uniform}
  \end{center}
\end{figure}

To show how our methodology scales with the number of packages, we repeated the process described above for different values of the number of packages $n$. As Table~\ref{T.improvements} shows, in all cases the improvement of the mixed strategy compared to the van-only strategy was between 31\% and 33\%, indicating the consistency and scalability of our proposed strategy.

\begin{table}[h]
\centering 
   \caption{Improvements with different number of packages}  \label{T.improvements}
 \begin{tabular}{ccc}  
 \hline
     \emph{Number of packages} & \emph{Optimal incentive rate} & \emph{Improvement} \\
     \hline
       600  & 1.22 &31.93\%\\
       1000 & 1.17&32.01\%\\ 
       1500 & 1.15&32.15\%\\
       2000 & 1.13&32.40\%\\
       3000 & 1.11&32.70\%\\
       \hline
   \end{tabular}
\end{table}

\subsubsection{Case 2: Packages distributed in three clusters over a square region}

For this numerical experiment we consider a more realistic scenario where packages are clustered within the service region. To simulate the package destinations we generated 2000 data points according to the following method:

\begin{itemize} 
\item 500 packages uniformly distributed in a $5\times5$ square;
\item 700 packages uniformly distributed in the ellipse $\frac{(x-1.5)^2}{1.2^2}+\frac{(y-4)^2}{1^2}=1$;
\item 500 packages uniformly distributed in the ellipse $\frac{(x-3.8)^2}{0.8^2}+\frac{(y-3.3)^2}{1.2^2}=1$;
\item 300 packages uniformly distributed in the ellipse $\frac{(x-2.5)^2}{1.2^2}+\frac{(y-1.4)^2}{1^2}=1$.
\end{itemize}

Figure~\ref{F.obj_cluster} plots the objective function \eqref{eq:Optimization} for different realizations of the package destinations ${\bf x}^{(n)}$. Note that since the data is clustered around three centers, the typical distance between points is smaller, which yields a smaller expected cost for the same number of packages as in Case 1. 

\begin{figure}[t]
\begin{center}
  \includegraphics[width=8cm]{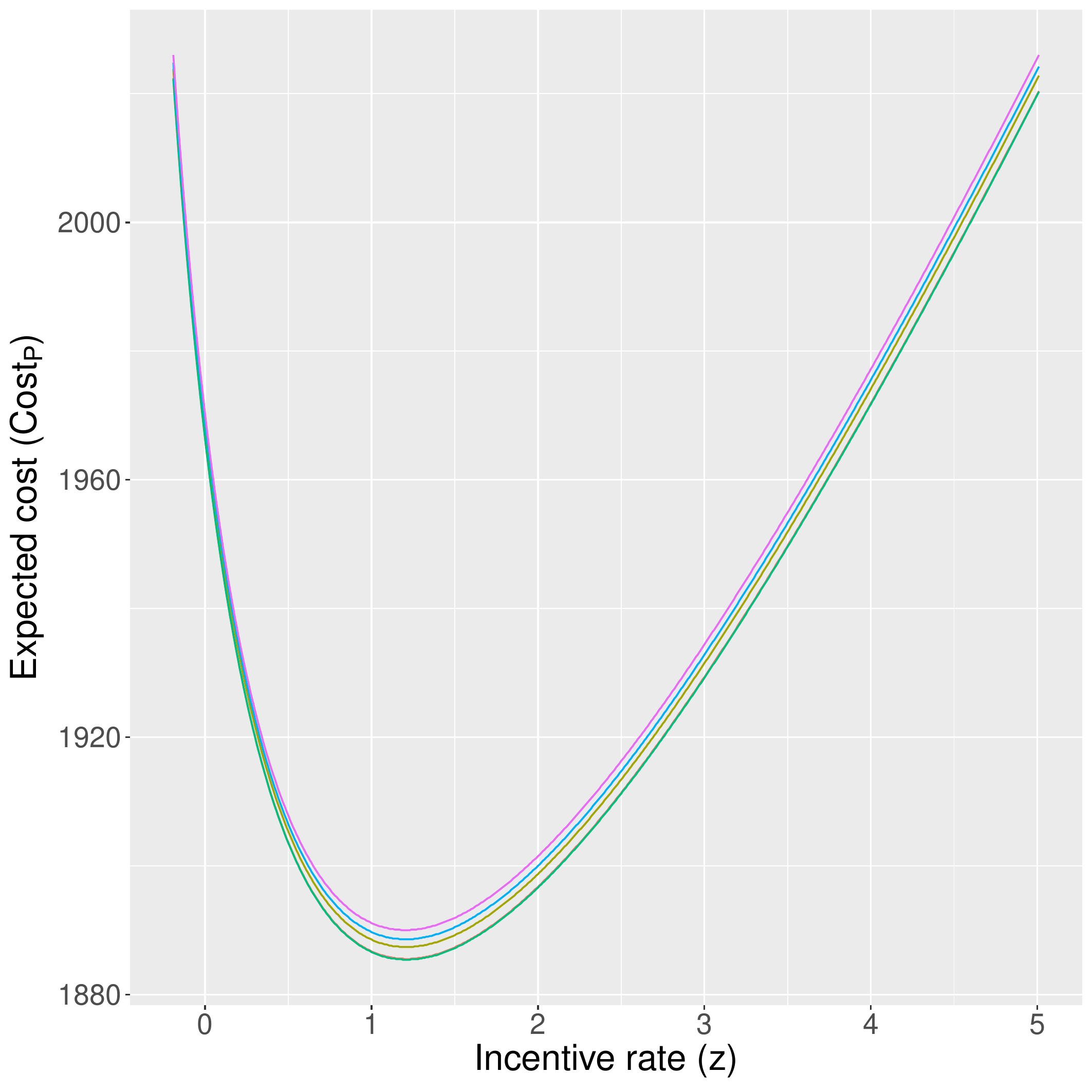}
  \caption{Objective function \eqref{eq:Optimization} for 2000 packages distributed in three clusters. Different colors represent different realizations of package destinations.}
\label{F.obj_cluster}
  \end{center}
\end{figure}

Following the same methodology as in Case 1, we computed optimal TSP and CVRP routes on the same single realization of package destinations, as depicted by Figures~\ref{F.tsp_cluster} and \ref{F.vrp_cluster}, respectively. The total lengths for the routes were L(TSP) = 176.61 miles and L(CVRP) = 193.23 miles, respectively. The total cost for the van-only strategy was computed to be
$$\text{L(CVRP(${\bf x}^{(n)}$))} \left( \zeta_V + \frac{h_V}{v_V}  \right)  + nh_V \tau_V = 193.23 (2.309) + 2000 (42.389) (97/3600) = \$ 2730.46.$$
We also computed the optimal solution $z^* = 1.21$ to \eqref{eq:Optimization} and the corresponding rate $\lambda(z^*)$, and used it to simulate the package pick up process. A single simulation run yielded a total of 531 leftover packages, an optimal CVRP route of length 81.36 miles, and a total cost of \$1829.31 (see Figure~\ref{F.cluster_left}). Therefore, the improvement of the mixed strategy compared to the van-only strategy was 33.00\%.

\begin{figure}[t]
\centering
\begin{subfigure}[h]{.48\textwidth}
\centering
\includegraphics[width=\linewidth,height=6cm]{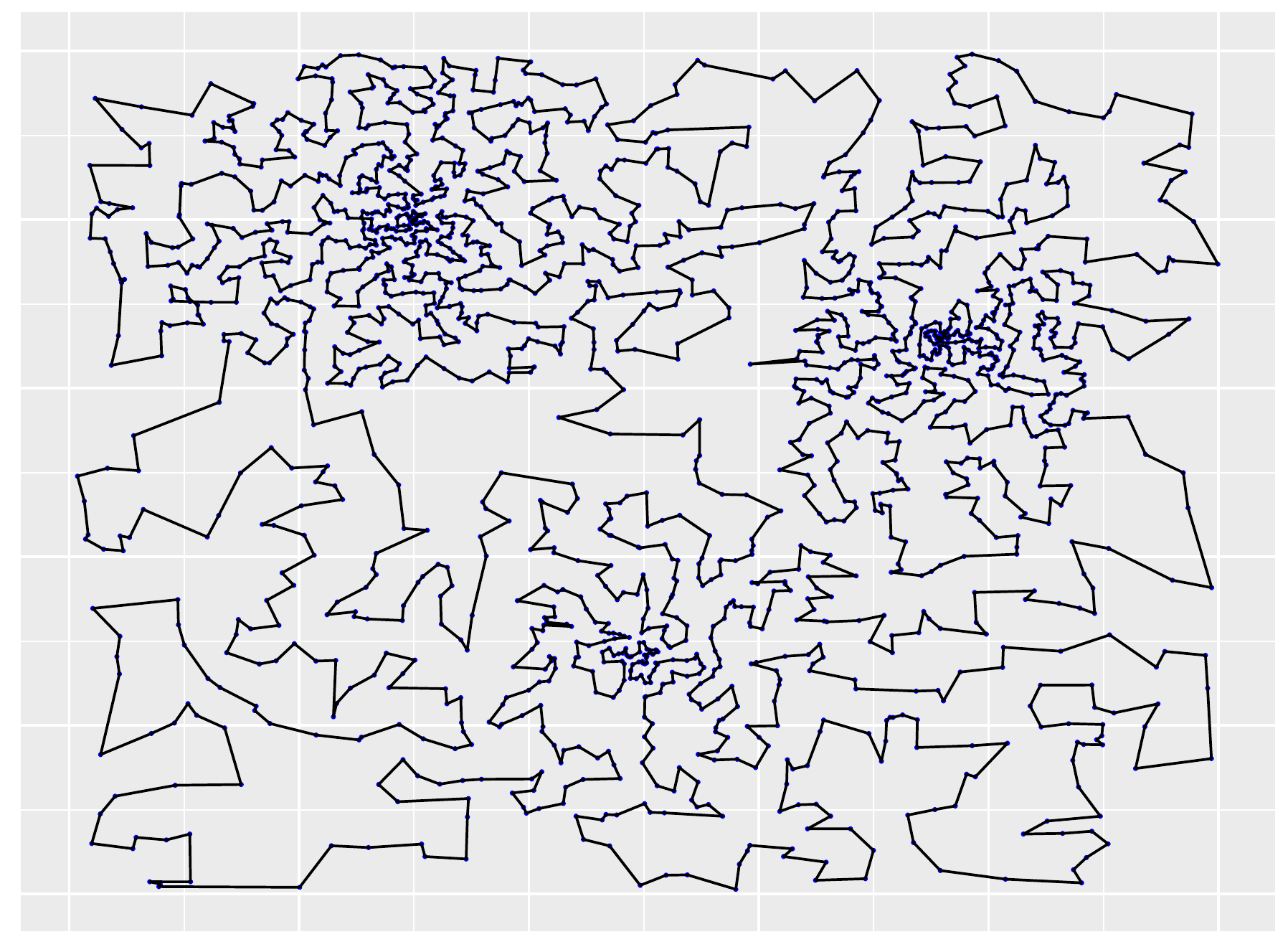}
\caption{Optimal TSP route solved using Concorde in less than 3 seconds. Length = 176.61 miles. This computation is needed for the mixed strategy.}
 \label{F.tsp_cluster}
\end{subfigure}
\begin{subfigure}[h]{.48\textwidth}
\centering
\includegraphics[width=\linewidth,height=6cm]{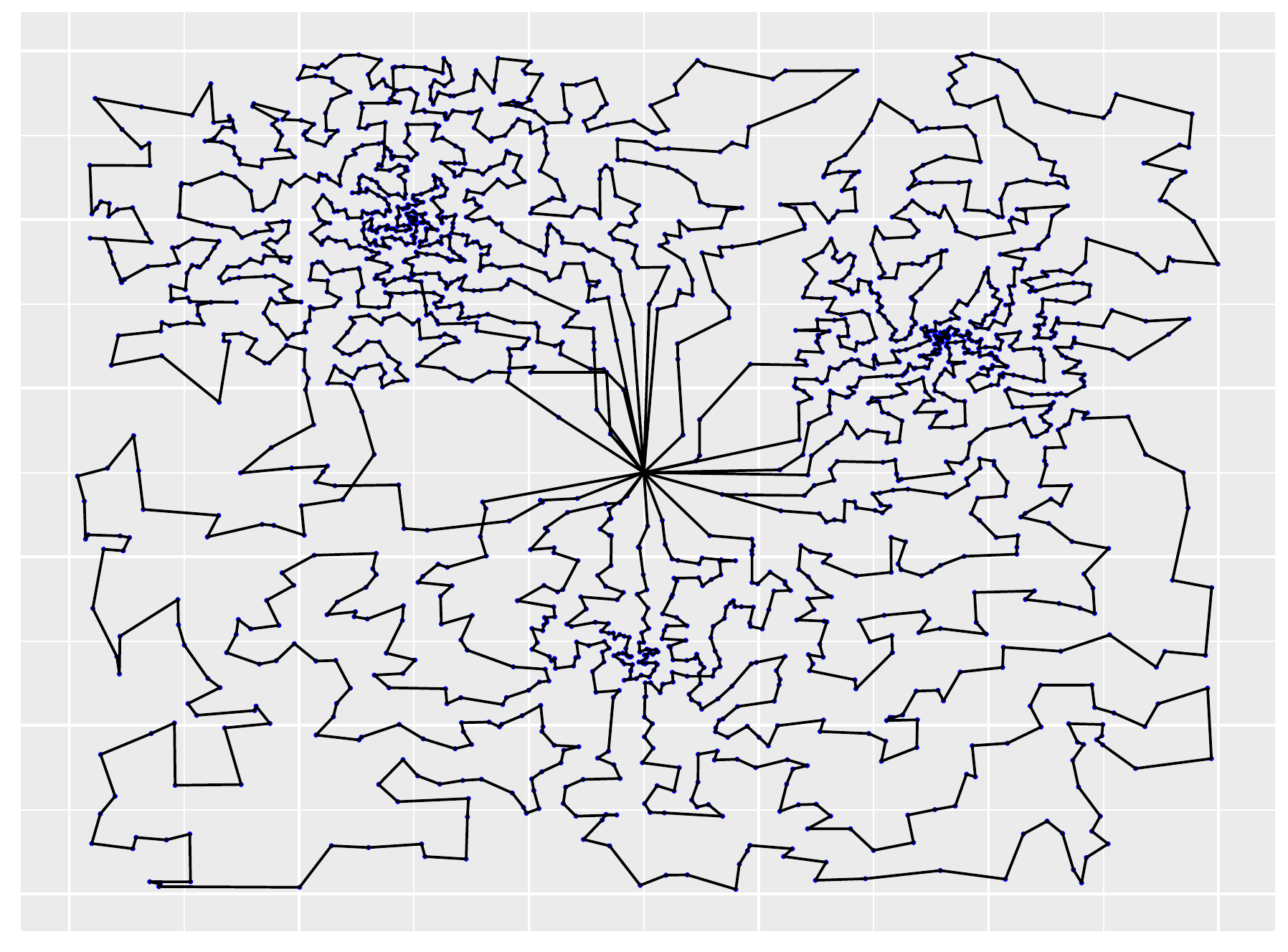}
\caption{Optimal CVRP route solved using VRPH1.0.0 in less than 60 seconds. Length = 193.23 miles. This computation is needed for the van-only strategy.}\label{F.vrp_cluster}
\end{subfigure}
\caption{Optimal TSP and CVRP routes for 2000 packages distributed in three clusters over a $5\times 5$ square miles region.}
\end{figure}

\begin{figure}[h ]
\begin{center}
  \includegraphics[width=8cm,height = 6cm]{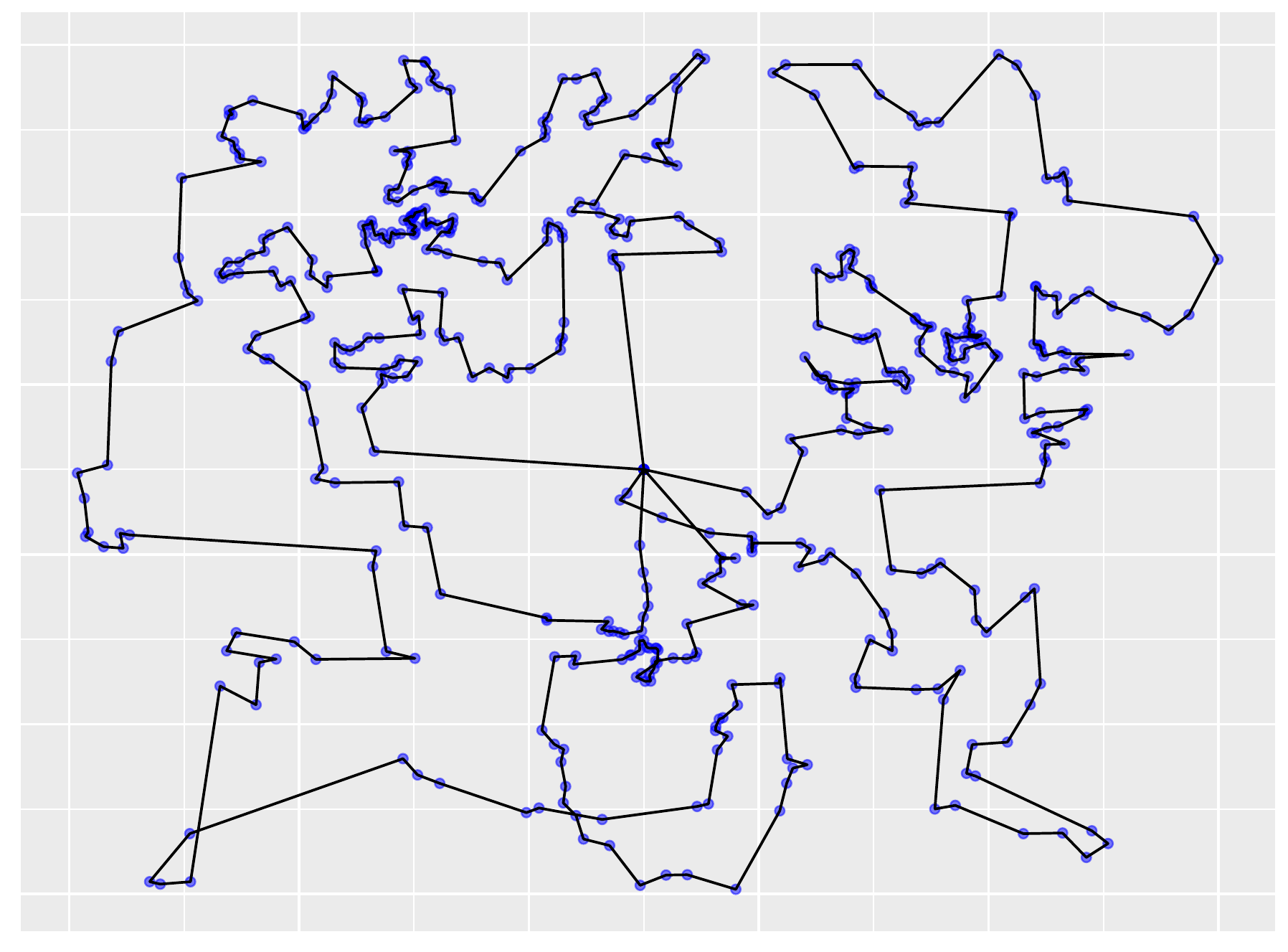}
  \caption{Optimal CVRP route for 531 leftover packages solved using VRPH 1.0.0. Length = 81.36 miles.}
  \label{F.cluster_left}
  \end{center}
\end{figure}

\subsection{Neighboring distance in the TSP route} \label{SS.TSPneighbor}

To verify the assumption that neighboring packages along an optimal TSP route are not too far away from each other, we did some numerical experiments and computed the empirical distribution of the distance between adjacent packages. We generated a total of 1000 points (package destinations) over a $5\times5$ square region, first using a uniform distribution as in Case 1 and then using three clusters as in Case 2. To compute an optimal TSP route we used Concorde. For each realization we computed the empirical density function of the the distance between adjacent packages, i.e., if we let $\boldsymbol{\xi}_i \in [0,5] \times [0,5]$ denote the location of the $i$th package along the optimal TSP route over the points $\{{\bf x}_1, \dots, {\bf x}_n\}$ (choosing randomly the first one), we computed
\begin{equation} \label{eq:DensityFunction}
f_n(x) = \frac{1}{n} \sum_{i=1}^n 1( d(\boldsymbol{\xi}_{i-1}, \boldsymbol{\xi}_i) \in dx),
\end{equation}
where $d$ is the $L_1$ distance.  Then, we ran 50 independent realizations of the 1000 points and plotted the average
$$g_n(x) = \frac{1}{50} \sum_{j=1}^{50} f_n^{(j)}(x),$$
where $f_n^{(j)}(x)$ is computed as in \eqref{eq:DensityFunction} on the $j$th run. Note that $g_n$ is an estimator for the density function of the distance between two neighbors, uniformly chosen at random, along the optimal TSP route.

The results are depicted in Figures~\ref{tsp_density_uniform} and \ref{tsp_density_clusters}.

\begin{figure}[h]
\centering
\begin{subfigure}[h]{.49\textwidth}
\centering
\begin{picture}(200,150)
\put(0,15){\includegraphics[scale=0.65, bb = 25 24 340 270, clip]{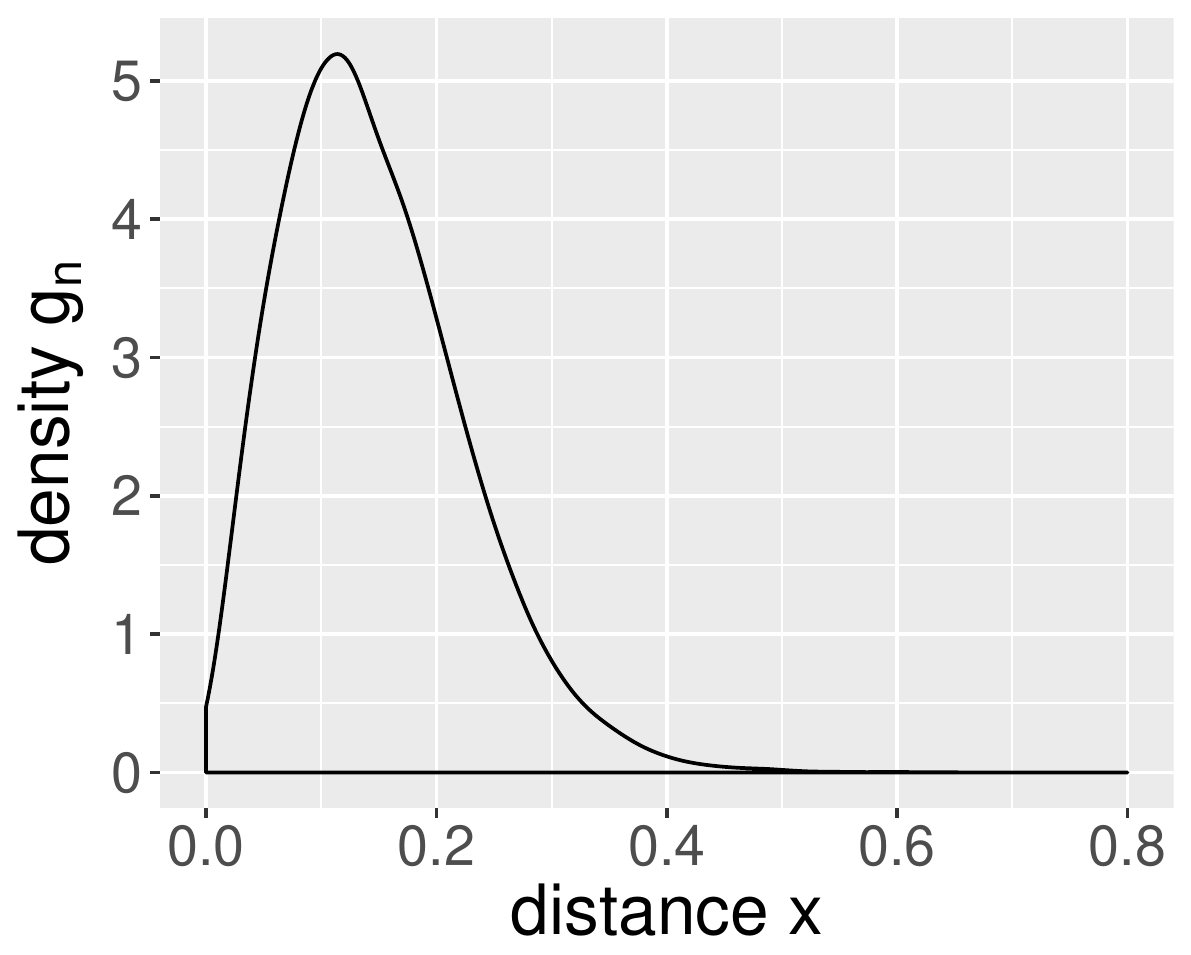}}
\put(80,3){Distance $x$}
\put(-12,60){\rotatebox{90}{Density $g_n(x)$}}
\end{picture}
\caption{Density function of neighboring distance $g_n$ along an optimal TSP route when packages are uniformly distributed.}
 \label{tsp_density_uniform}
\end{subfigure}
\begin{subfigure}[h]{.49\textwidth}
\centering
\begin{picture}(200,150)
\put(0,15){\includegraphics[scale=0.65, bb = 25 24 340 270, clip]{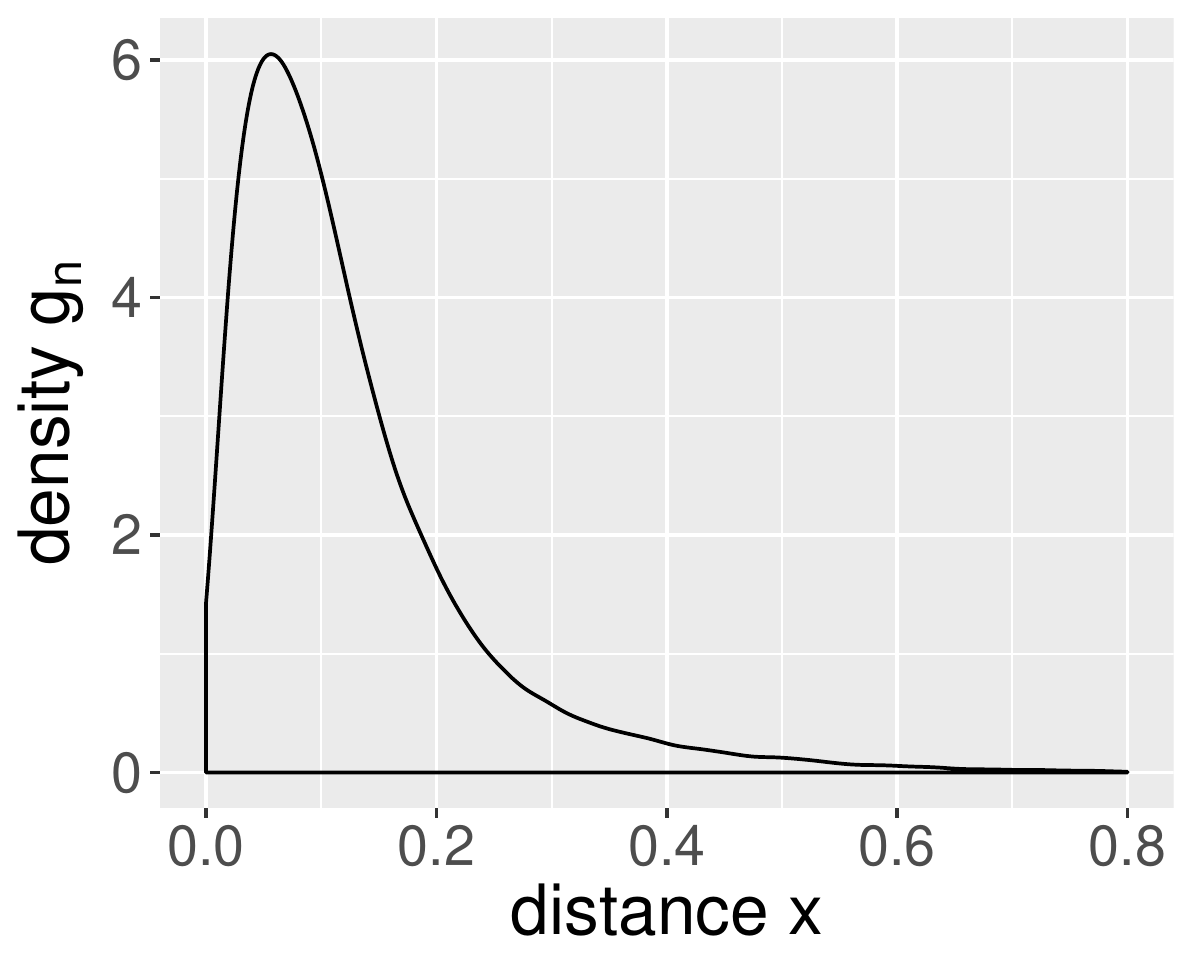}}
\put(80,3){Distance $x$}
\put(-12,60){\rotatebox{90}{Density $g_n(x)$}}
\end{picture}
\caption{Density function of neighboring distance $g_n$ along an optimal TSP route when packages are distributed in three clusters.}\label{tsp_density_clusters}
\end{subfigure}
\caption{Empirical density function of neighboring distance for both scenarios.}\label{tsp_neighbor}
\end{figure}

For uniformly distributed packages, the mean neighboring distance was 0.146, the third quantile was
0.193, and 95\% of the data lied below 0.290, with a standard deviation of 0.079. For the case of three clusters, the mean neighboring distance was 0.125, the third quantile was 0.108, and 95\% data lied below 0.334, with a standard deviation of 0.108. Note that, as expected, the mean in the three cluster case is smaller than that in the uniform case, but the standard deviation is larger. In both scenarios, almost all packages are very close to each other.

 \section{Appendix}

The following result taken from  \cite{haimovich1985bounds} (see also \cite{daganzo2005logistics}) provides upper and lower bounds for the length of an optimal CVRP route to deliver $n$ packages with arbitrary destinations ${\bf x}^{(n)} = \{ {\bf x}_1, \dots, {\bf x}_n\} \subseteq \mathbb{R}^2$.

\begin{theorem} \label{T.Daganzo}
For any set ${\bf x}^{(n)}$ of demand points serviced by a fleet of vehicles with capacity $V$ that originate from a single depot, we have
$$\max\left\{\frac{2n\bar{r}(n)}{V},L(\text{TSP}({\bf x}^{(n)}))\right\}\leq L(\text{CVRP}({\bf x}^{(n)}))\leq 2\left(\frac{n}{V}+1\right)\bar{r}(n)+L(\text{TSP}({\bf x}^{(n)})),$$
where $\bar{r}(n)$ is the empirical average distance, i.e., $\bar{r}(n)=\frac{1}{n}\sum_{i=1}^n r_i$.
\end{theorem}

The following result taken from \cite{karloff1989long} (see also \cite{goddyn1990quantizers}) gives an asymptotic upper bound for the length of an optimal TSP route through the points ${\bf x}^{(n)}$. 

\begin{theorem}\label{T.TSPbound}
For any sequence of points ${\bf x}^{(n)}$ contained in a compact region $R \subseteq \mathbb{R}^2$, and any of the $L_p$ distances on $\mathbb{R}^2$, there  exists a constant $\alpha_{TSP}(p)$ satisfying 
$$\limsup_{n\rightarrow\infty}\frac{L(TSP({\bf X}^{(n)}))}{\sqrt{nA}}\leq \alpha_{TSP}(p),$$
where $A$ is the area of $R$. 
\end{theorem}

Note that the statement appearing in \cite{goddyn1990quantizers,karloff1989long} is stated only for $p = 2$, however,  the general case can be obtained from $\alpha_\text{TSP}(2)$ by noting that $d_q({\bf x}, {\bf y}) \leq d_2({\bf x}, {\bf y}) \leq 2^{1/2-1/p} d_p({\bf x}, {\bf y})$ for $1 \leq p \leq 2 \leq q$.

\section*{Acknowledgments}

Mariana Olvera-Cravioto was funded through NSF grant CMMI-1537638.

\bibliography{packing}
\bibliographystyle{plain}

\end{document}